\documentclass{article}

\usepackage{arxiv}

\usepackage[utf8]{inputenc} % allow utf-8 input
\usepackage[T1]{fontenc}    % use 8-bit T1 fonts
\usepackage{hyperref}       % hyperlinks
\usepackage{url}            % simple URL typesetting
\usepackage{booktabs}       % professional-quality tables
\usepackage{amsfonts}       % blackboard math symbols
\usepackage{nicefrac}       % compact symbols for 1/2, etc.
\usepackage{microtype}      % microtypography
\usepackage{lipsum}

\usepackage{amsmath}

\usepackage{color}
\usepackage{amsthm}
\theoremstyle{definition}
\newtheorem{definition}{Definition}[section]
\newtheorem{prop}{Proposition}[section]
\newtheorem{lemma}{Lemma}[section]
\newtheorem{thm}{Theorem}[section]

\newtheorem{assumption}{Assumption}[section]
\newtheorem{corr}{Corollary}[section]
\theoremstyle{remark}
\newtheorem{remark}{Remark}
\title{The Galerkin analysis for the random periodic solution of semilinear stochastic evolution equations}

\author{
Yue Wu$^{*}$\\
  Department of Mathematics and Statistics\\
   University of Strathclyde\\
    Glasgow, G1 1XQ, UK;\\
    Mathematical Institute\\
    University of Oxford\\
    Oxford, UK\\
  \texttt{yue.wu@strath.ac.uk$^{*}$} \\
   \And
Chenggui Yuan\\
  Department of Mathematics\\
  Swansea University\\
  Swansea, SA1 8EN, UK \\
  \texttt{c.yuan@swansea.ac.uk} \\
}

\begin{document}
\maketitle

\begin{abstract}
In this paper we study the numerical method for approximating the random periodic solution of semilinear stochastic evolution equations.  The main challenge lies in proving a convergence over an infinite time horizon while simulating infinite-dimensional objects.
We first show the existence and uniqueness of the random periodic solution to the equation as the limit of the pull-back flows of the equation,  and observe that its mild form is well-defined in the intersection of a family of decreasing Hilbert spaces.
Then we propose a Galerkin-type exponential integrator scheme and establish its convergence rate of the strong error to the mild solution,   where the order of convergence directly depends on the space (among the family of Hilbert spaces) for the initial point to live.  We finally conclude with a best order of convergence that is arbitrarily close to 0.5. 
\end{abstract}

% keywords can be removed

\section{Introduction}
The random periodic solution is a new concept to characterize the presence of random periodicity in the long run of some stochastic systems. On its first appearance in \cite{zhao2009}, the authors gave the definition of the random periodic solutions of random dynamical systems and showed the existence of such periodic solutions for a $C^1$ perfect cocycle on a cylinder. This is followed by another seminal paper \cite{rpssde2011}, where the authors not only defined the random periodic solutions for semiflows but also provided a general framework for its existence. Namely, instead of following the traditional geometric method of establishing the Poincar\'e mapping, a new analytical method for coupled infinite horizon forward-backward integral equations was introduced. This pioneering study boosts a series of work, including 
the existence of random periodic solutions to stochastic partial differential equations (SPDEs) \cite{rpsspde2011}, the existence of anticipating random periodic solutions \cite{wu2016,wu2018}, periodic measures \cite{feng2020}, etc. 

Let us recall the definition of the random periodic solution for stochastic semi-flows given in \cite{rpssde2011}. Let $H$ be a separable Banach space. Denote by $(\Omega,{\cal F},\mathbb{P},(\theta_s)_{s\in \mathbb{R}})$ a metric dynamical system and $\theta_s:\Omega\to\Omega$ is  assumed to be measurably invertible for all $s\in \mathbb{R}$.
Denote $\Delta:=\{(t,s)\in \mathbb{R}^2, s\leq t\}$. Consider a stochastic semi-flow $u: \Delta \times \Omega\times H\to H$, which satisfies the following standard condition
\begin{eqnarray}{\label{16}}
u(t,r,\omega)=u(t,s,\omega)\circ u(s,r,\omega),\ \ {\rm for\ all } \ r\leq s\leq t,\ r, s,t\in \mathbb{R},\ \mbox{for }a.e.\ \omega\in\Omega.
\end{eqnarray}
We do not assume the map $u(t,s,\omega): H\to H$ to be invertible for $(t,s)\in \Delta,\ \omega \in \Omega$. 

\begin{definition}
\label{feng-zhao1}
A {\bf random periodic path of period $\tau>0$ of the semi-flow} $u: \Delta\times \Omega\times H\to H$
 is an ${\cal F}$-measurable
map $y:\mathbb{R}\times \Omega\to H$ such that for a.e. $\omega\in \Omega$
\begin{eqnarray}\label{eqn:def}
\Big\{\begin{array}{l}u(t,s, y(s,\omega), \omega)=y(t,\omega),\ \ \forall t\geq s\\

y(s+\tau,\omega)=y(s, \theta_\tau \omega),\ \ \forall s\in  \mathbb{R}.
                                     \end{array}
\end{eqnarray}
\end{definition}  

Note that Definition \ref{feng-zhao1} covers both the deterministic periodic path and the random fixed point (c.f. \cite{arnold}), also known as stationary point as its special cases.  To see the latter one,  one may assume \eqref{eqn:def} holds for any $\tau>0$,  and define $\hat{y}(\theta_t\omega)=y(0,\theta_t\omega)$ for $t>0$,  then one can conclude that $u(t,0, \hat{y}(\omega), \omega)=\hat{y}(\theta_t\omega)$ from \eqref{eqn:def},  which coincides with the definition of random fixed point (also termed as the \emph{stationary solution}) given in \cite{arnold}.  A well-known example for stationary solution is given by $Y(\omega)=\int_{-\infty}^{0}e^s \mathrm{d}W(s)$,  for the one-dimensional random dynamical system $\phi(t,\omega)x=xe^{-t}+\int_0^te^{-(t-s)}\mathrm{d}W(s,\omega)$ generated from the following Ornstein-Uhlenbeck process:
\begin{equation}\label{eqn:OU}
d y(t)=-y(t)\mathrm{d}t+dW(t), \ \ y(0)=x\in \mathbb{R}, \ \ t>0,
\end{equation}
where $W: (t,\omega) \mapsto W(t,\omega)$ is a one-dimensional two-sided Wiener process on $(\Omega, \mathcal{F}, \mathbb{P})$,  and as a convention,   $\omega$ is usually hidden in the notation $W(s,\omega)$.
One can verify that $\phi(t,\omega)Y(\omega)=Y(\theta_t \omega)$.  If in addition,  we add a periodic drift term to Eqn.  \eqref{eqn:OU} such that it reads as
\begin{equation}\label{eqn:OU2}
d y(t)=(-y(t)+\sin(t))\mathrm{d}t+dW(t), \ \ y(s)=x\in \mathbb{R}, \ \ t>0,
\end{equation}
then it is not hard to see the semiflow for \eqref{eqn:OU2} is given by $\varphi(t,s,x,\omega):=xe^{-(t-s)}+\int_s^te^{-(t-r)}\sin(r)\mathrm{d}r+\int_s^te^{-(t-r)}\mathrm{d}W(r)$.  Now define $Y(t,\omega)=\int_{-\infty}^t e^{-(t-s)}\sin(s)\mathrm{d}s+\int_{-\infty}^t e^{-(t-s)}\mathrm{d}W(s)$.  One can verify that $Y(t,\omega)=\varphi(t,s,\omega)Y(s,\omega)$ and $Y(t+2\pi,\omega)=Y(t,\theta_{2\pi}\omega)$. Indeed,
\begin{align*}
Y(t+2\pi,\omega)&=\int_{-\infty}^{t+2\pi} e^{-(t+2\pi-s)}\sin(s)\mathrm{d}s+\int_{-\infty}^{t+2\pi} e^{-(t+2\pi-s)}\mathrm{d}W(s,\omega)\\
&=\int_{-\infty}^{t} e^{-(t-\hat{s})}\sin(\hat{s}+2\pi)\mathrm{d}\hat{s}+\int_{-\infty}^{t} e^{-(t-\hat{s})}\mathrm{d}W(\hat{s}+2\pi,\omega)\\
&=\int_{-\infty}^{t} e^{-(t-\hat{s})}\sin(\hat{s})\mathrm{d}\hat{s}+\int_{-\infty}^{t} e^{-(t-\hat{s})}\mathrm{d}(W(\hat{s}+2\pi,\omega)-W(2\pi,\omega))\\
&=\int_{-\infty}^{t} e^{-(t-\hat{s})}\sin(\hat{s})\mathrm{d}\hat{s}+\int_{-\infty}^{t} e^{-(t-\hat{s})}\mathrm{d}W(\hat{s},\theta_{2\pi}\omega)=Y(t,\theta_{2\pi}\omega)
\end{align*}
%$Y(t,\omega)=\int_{-\infty}^t e^{-(t-s)/2+(W(t,\omega)-W(s,\omega))}\sin(s)\mathrm{d}s$,  one can verify that $Y(t+2\pi,\omega)=Y(t,\theta_{2\pi}\omega)$.  Indeed,
%\begin{align*}
%Y(t+2\pi,\omega)&=\int_{-\infty}^{t+2\pi} e^{-(t+2\pi-s)/2+(W(t+2\pi,\omega)-W(s,\omega))}\sin(s)\mathrm{d}s\\
%&=\int_{-\infty}^{\hat{s}} e^{-(t+\hat{s})/2+(W(t+2\pi,\omega)-W(\hat{s}+2\pi,\omega))}\sin(\hat{s}+2\pi)\mathrm{d}\hat{s}\\
%&=\int_{-\infty}^{\hat{s}} e^{-(t+\hat{s})/2+(W(t+2\pi,\omega)-W(2\pi,\omega)+W(2\pi,\omega)-W(\hat{s}+2\pi,\omega))}\sin(\hat{s}+2\pi)\mathrm{d}\hat{s}\\
%&=\int_{-\infty}^{\hat{s}} e^{-(t+\hat{s})/2+(\theta_{2\pi})(W(t)-W(\hat{s}))}\sin(\hat{s})\mathrm{d}\hat{s}=Y(t+2\pi,\theta_{2\pi}\omega),
%\end{align*}
where we use in the last two lines the measure preserving property of Wiener process stated in Assumption \ref{as:preserve}.  Therefore $Y$ is a random periodic path for semiflow $\varphi$ generated from SDE \eqref{eqn:OU2}.

In general, random periodic solutions cannot be solved explicitly.  Even for the simple case as we showcased in Eqn. \eqref{eqn:OU2},  one relies on numerical approaches to simulate the random periodic path $Y$.  For the dissipative system generated from some SDE with a global Lipchitz condition, the convergences of a forward Euler-Maruyama method and a modified Milstein method to the random period solution have been investigated in \cite{rpsnumerics2017}. For SDEs with a monotone drift condition, one benefits a flexible choice of stepsize from applying the implicit method instead \cite{wu2021}. Each of these numerical schemes admits their own random periodic solution, which approximates the random periodic solution of the targeted SDE as the stepsize decreases. The main challenge lies in proving a convergence of an infinite time horizon. In this paper, we consider approximating the random periodic trajectory of SPDEs, where, we encounter an additional obstacle of simulating infinite-dimensional objects.  For this, we employ the spectral Galerkin method (c.f. \cite{spectrum}) for spatial dimension reduction, construct a discrete exponential integrator scheme based on the spatial discretization and conclude the existence and uniqueness of random periodic solution from the discrete scheme.  To the best of our knowledge,  this is the first study that works on the numerical analysis (Galerkin analysis) of random periodic solutions for SEEs.  The Galerkin-type method has been intensively used to simulate solutions of
parabolic SPDEs over finite-time horizon \cite{grecksch1996, gyongy1998,gyongy1999, kruse2014c,kruse2014o, kruse2019}, and it is recently applied to approximate stationary distributions for SPDEs \cite{bao2014}. For the error analysis of both strong and weak approximation of semilinear stochastic evolution equations (SEEs) through Galerkin approximation, we refer
the reader to the monograph \cite{kruse2014book}.  

 Let $(H, (\cdot, \cdot), \| \cdot \|)$ and $(U,
(\cdot, \cdot)_U, \| \cdot \|_U)$ be two separable $\mathbb{R}$-Hilbert spaces. For a
given $t_0,T \in (-\infty,\infty)$ with $t_0<T$ we denote by $(\Omega, \mathcal{F}, (\mathcal{F}_t)_{t \in [t_0,T]},
\mathbb{P})$ a filtered probability space satisfying the usual conditions. By
$(W(t))_{t \in [t_0,T]}$ we denote an $(\mathcal{F}_t)_{t \in [t_0,T]}$-Wiener
process on $U$ with associated covariance operator $Q \in \mathcal{L}(U)$,
which is not necessarily assumed to be of finite trace.  Denote by $\mathcal{L}_2^0=\mathcal{L}_2^0(H) =\mathcal{L}_2(Q^{\frac{1}{2}}(U),H)$ the set of
all Hilbert-Schmidt operators from $Q^{\frac{1}{2}}(U)$ to $H$. 

Our goal is to study and approximate the random periodic mild solution to SEEs of the form
\begin{align}
  \label{eq:SPDE}
  \begin{cases}
    \mathrm{d}{X^{t_0}_t} = 
    \big[ -A X^{t_0}_t + f(t,X^{t_0}_t) \big] \mathrm{d}{t}+g(t,X^{t_0}_t) \mathrm{d}{W(t)},& \quad \text{for } t \in (t_0,T],\\
    X^{t_0}_{t_0} = \xi.& 
  \end{cases}
\end{align}
Throughout the paper, we impose the following essential assumptions.
\begin{assumption}
  \label{as:A}
  The linear operator $A \colon \text{dom}(A) \subset H \to H$ is densely defined,
  self-adjoint, and positive definite with compact inverse.
\end{assumption}

Assumption \ref{as:A} implies the existence of a positive, increasing
sequence $(\lambda_i)_{i\in \mathbb{N}} \subset \mathbb{R}$ such that 
$0<\lambda_1 \le \lambda_2 \le \ldots $ with $\lim_{i\to 
\infty}\lambda_i = \infty$, and of an orthonormal basis $(e_i)_{i\in
\mathbb{N}}$ of $H$ such that $A e_i = \lambda_i  e_i$ for every $i \in
\mathbb{N}$. Indeed we have that 
  $$\text{dom}(A):=\{x\in H:\sum_{n=1}^\infty \lambda_n^2(x,e_n)^2<\infty \}.$$ In addition, it also follows from Assumption~\ref{as:A} that $-A$ is the
infinitesimal generator of an analytic semigroup $(S(t))_{t \in [0,\infty)}
\subset \mathcal{L}(H)$ of contractions. More precisely, the family
$(S(t))_{t \in [0,\infty)}$ enjoys the properties
\begin{align*}
  S(0) &= \mathrm{Id} \in \mathcal{L}(H),\\
  S(s + t) &= S(s) \circ S(t) = S(t) \circ S(s), 
  \quad \text{for all } s,t \in [0,\infty),
\end{align*}
and
\begin{align}
  \label{eq:stab_S}
  \sup_{t \in [0,\infty)} \| S(t) \|_{\mathcal{L}(H)} \le 1.  
\end{align}

Further, let us introduce fractional powers of $A$, which are used to measure
the (spatial) regularity of the mild solution \eqref{eq:mild}. For any $r\in
[-1,1]$ we define the operator $A^{\frac{r}{2}} \colon \text{dom}(A^{\frac{r}{2}}) =
\{x\in H \, : \, \sum_{j=1}^{\infty} \lambda_j^r (x,e_j)^2 < \infty \}
\subset H \to H$ by 
\begin{equation}
  \label{eqn:fractional_r}
  A^{\frac{r}{2}} x
  := \sum_{j=1}^{\infty} \lambda_j^{\frac{r}{2}} (x,e_j) e_j,
  \quad \text{for all } x \in \text{dom}(A^{\frac{r}{2}}).
\end{equation}
Then, by setting $(\dot{H}^r,(\cdot,\cdot)_r, \|\cdot\|_r) :=
(\text{dom}(A^{\frac{r}{2}}), (A^{\frac{r}{2}} \cdot, A^{\frac{r}{2}}\cdot),
\|A^{\frac{r}{2}}\cdot\|)$, we obtain a family of
separable Hilbert spaces. Clearly, for any $0\leq r_1<r_2\leq 1$, we have that $\text{dom}(A)\subset\dot{H}^{r_2}\subset \dot{H}^{r_1}\subset H$.

\begin{assumption}
  \label{as:ini}
  The initial value
  $\xi \colon \Omega \to H$ satisfies $\xi \in L^{2}(\Omega, \mathcal{F}_{t_0},\mathbb{P};
  H)$. Denote by $C_\xi$ a constant such that $\mathbb{E}[\|\xi\|^2]\leq C_\xi^2$.
\end{assumption}

\begin{assumption}
  \label{as:fg}
  The mappings $f \colon \mathbb{R} \times H \to H$ and $g \colon \mathbb{R} \times H \to \mathcal{L}^2_0$ are continuous and periodic in time with period $\tau$.
  Moreover, there exist $\kappa\in (0,1/2]$,  $C_f, C_g, C_{f,g} \in
  (0,\infty)$ such that 
  \begin{align*}
  & \| f(t,u_1) - f(t,u_2) \| \leq C_f \|u_1 - u_2\|,\qquad \| f(t_1,u) - f(t_2,u) \|\leq C_f (1+\| u\|)|t_1 -
    t_2|^{\kappa},\\
    &\|g(t,u_1)-g(t,u_2)\|_{\mathcal{L}^2_0}\leq C_g \| u_1-u_2\|,\qquad \| g(t_1,u) - g(t_2,u) \|_{\mathcal{L}^2_0}\leq C_g (1+\| u\|)|t_1 -
    t_2|^{\kappa},
  \end{align*}
  for all $u,u_1, u_2 \in H$ and $t,t_1,t_2 \in [0,\tau)$.
\end{assumption}
\begin{remark}
Indeed the condition on $f$ can be weakened  to a local Lipschitz condition for the existence and uniqueness of the random periodic solution.  However,  to show the continuity of the random periodic solution or to conduct the numerical analysis,  one still need Assumption \ref{as:fg}. 
\end{remark}
%\begin{remark}
%One will see that the H\"older continuity in temporal variable imposed on both diffusion and drift terms plays a role in the numerical analysis part only.  To be more specific,  though it does not affect at all any results regarding the existence and uniqueness of the random periodic solution from SEE (results in Section \ref{sec:existence}) or from its numerical scheme (all results prior to Section \ref{sec:convergence}),  it partly determines the order of convergence of the proposed numerical scheme.
%\end{remark}
\begin{remark}
One will see that the H\"older continuity in temporal variable imposed on both diffusion and drift terms plays an important role in the numerical analysis part.  To be more specific,  it partly determines the order of convergence of the proposed numerical scheme.
\end{remark}
\begin{remark} Note that the assumption on $g$ excludes identity in $\mathcal{L}(H)$.  One may refer to \cite{bao2014} for techniques handing a slight more general assumption on $g$,  which allows $g$ to be constant in $\mathcal{L}(H)$.  
\end{remark}

From Assumption~\ref{as:fg} we directly deduce a linear growth bound
of the form
\begin{align}
  \label{eq3:linear_f}
  \|f(t,u)\|+\|g(t,u)\|_{\mathcal{L}^2_0}\leq L_{f,g} ( 1 + \|u\| ),\quad \text{for all }
  t \in \mathbb{R},\, u \in H,
\end{align}
where $L_{f,g}:=\|f(0,0)\|+\|g(0,0)\|_{\mathcal{L}^2_0}+(C_f+C_g)(1+\tau^{\frac{1}{2}})$.

Under these assumptions the SEE \eqref{eq:SPDE} admits a unique \emph{mild solution} $X_{\cdot}^{t_0} \colon [t_0,T] \times
\Omega \to H$ such that it is uniquely determined by the
variation-of-constants formula (c.f. \cite{daprato1996})
\begin{align}
  \label{eq:mild}
  X^{t_0}_t(\xi) = S(t-t_0) \xi + \int_{t_0}^t S(t - s) f(s,X^{t_0}_s) \mathrm{d}{s} + \int_{t_0}^t S(t-s) g(s,X^{t_0}_s)
  \mathrm{d}{W(s)},
\end{align}
which holds $\mathbb{P}$-almost surely for all $t \in [t_0,T]$.

\subsection{The pull-back}
To ensure the existence of random periodic solution, we need some additional assumptions on the Wiener process and on $C_f$ and $\lambda_1$:
\begin{assumption}\label{as:lambda1}
  The constant $C_f$ in Assumption \ref{as:fg} and the eigenvalue $\lambda_1$ of $A$ satisfy $C_f <\lambda_1$.
\end{assumption}
\begin{assumption}
  \label{as:preserve}
  There exists a standard $\mathbb{P}$-preserving ergodic Wiener shift $\theta$ such that $\theta_s (\omega)(t)=W(t+s)-W(s)$ for $s,t\in \mathbb{R}$, ie, $$\mathbb{P}\circ (\theta_s W(t))^{-1}=\mathbb{P}\circ (W(t+s)- W(s))^{-1}.$$
\end{assumption}
%\begin{assumption}
%  \label{as:semiflow}
%  There exists a stochastic semi-flow $u: \Delta \times \Omega\times H\to H$ for SEE \eqref{eq:SPDE} and $u(t,s,\cdot,\omega): H\to H$ is continuous almost surely.
%\end{assumption}
Denote by $X^{-k\tau}_t(\xi,\omega)
$ the solution starting from time $-k\tau$. The uniform boundedness of $X^{-k\tau}_t(\xi,\omega)
$ in the $L^2$ sense can be guaranteed under Assumption \ref{as:A} to \ref{as:lambda1}. Further, under Assumption \ref{as:A} to Assumption \ref{as:preserve}, one is able to show that when $k\to \infty$, the
pull-back $X^{-k\tau}_t(\xi)$ has a unique limit $X^*
_t$ in $L^2(\Omega; H)$, moreover, $X^*
_t$ is the random periodic solution of SEE \eqref{eq:SPDE}, satisfying
\begin{equation}
  \label{eq:limit}
  X^*_t =  \int_{-\infty}^t S(t - s) f(s,X^{*}_s) \mathrm{d}{s} + \int_{-\infty}^t S(t-s) g(s,X^{*}_s) 
  \mathrm{d}{W(s)}.
\end{equation}
Surprisingly, the mild form \eqref{eq:limit} can be shown well-defined in $L^2(\Omega;\dot{H}^r)$ for any $r\in (0,1)$. More details about the proof can be found in Section \ref{sec:existence}. Besides, the continuity of $X^{-k\tau}_t(\xi,\omega)
$ is characterized in Section \ref{sec:existence} for error analysis in Section \ref{sec:numerical}. 
%
% with an additional assumption imposed:
%\begin{assumption}
%  \label{as:S}
%  There exists a positive $\alpha$ such that $\| S(t) \|_{\mathcal{L}(H)} \leq e^{-\alpha t}$.
%\end{assumption}

\subsection{The Galerkin approximation}
Next, we formulate the assumptions and notations on the spatial discretization. To this end, define finite-dimensional subspaces $H_n$ of $H$ spanned by the first $n$ eigenfunctions of the basis, ie, $H_n:=\{e_1,\ldots,e_n\}$, and let $P_n:H\to H_n$ be the orthogonal projection. Note that $H_n \subset \dot{H}^r$ for any $r\in \mathbb{R}$. By doing this, we are able to further introduce the following notations: $A_n =P_n A\in \mathcal{L}(H_n)$, $S_n(t)=P_nS(t)\in \mathcal{L}(H_n)$, $f_n=P_n f:\mathbb{R}\times H_n \to H_n$ and $g_n=P_ng:\mathbb{R}\times H_n \to \mathcal{L}_0^2(H_n)$. Then the Galerkin approximation to \eqref{eq:SPDE} can be formulated as follows
\begin{align}
  \label{eq:G}
  \begin{cases}
    \mathrm{d}{X^{n,t_0}_t} = 
    \big[ -A_n X^{n,t_0}_t + f_n(t,X^{n,t_0}_t) \big] \mathrm{d}{t}+g_n(t,X^{n,t_0}_t) \mathrm{d}{W(t)},& \quad \text{for } t \in (t_0,T],\\
    X^{n,t_0}_{t_0} = P_n\xi.& 
  \end{cases}
\end{align}
Applying the spectral Galerkin method results in
a system of finite dimensional stochastic differential equations. Note that for $x,y\in H_n$, we have that $A_nx=Ax$, $S_n(t)x=S(t)x$ and $\big(x, f_n(t,y)\big) =\big( x,f(t,y)\big)$. 
\begin{remark}
It is easy to see there exists an isometry between $H_n$ and $\mathbb{R}^n$.
An simply calculation leads to the existence of a unique strong solution to \eqref{eq:G}. The uniform boundedness of $X^{n,-k\tau}_t$ as well as the existence of the random periodic solution to \eqref{eq:G} are simple consequences of the corresponding properties of $X^{-k\tau}_t$.
\end{remark}
Let us fix an equidistant partition $\mathcal{T}^h:=\{jh,\ j\in \mathbb{Z} \}$ with stepsize $h\in (0,1)$. Note that $\mathcal{T}^h$ stretch along the real line because we are dealing with an infinite time horizon problem. Then to simulate the solution to \eqref{eq:G} starting at $-k\tau$, the \emph{discrete exponential integrator scheme} on $\mathcal{T}^h$ is given by the recursion 
\begin{align}
  \label{eq:RandM}
  \begin{split}
    \hat{X}_{-k\tau+(j+1)h}^{n,-k\tau} =& S_n(h)\Big(\hat{X}_{-k\tau+jh}^{n,-k\tau} +h f_n\big(-k\tau+jh,   \hat{X}_{-k\tau+jh}^{n,-k\tau} \big)+ 
    g_n(-k\tau+jh,\hat{X}_{-k\tau+jh}^{n,-k\tau} )\Delta W_{-k\tau+jh}\Big),
  \end{split}
\end{align}
for all $j \in \mathbb{N}$, where the initial value $\hat{X}_{-k\tau}^{n,-k\tau}  = P_n\xi$. 
Moreover, if we define
$$\bar{X}_t^{-k\tau} =\hat{X}_{-k\tau+jh}^{n,-k\tau}\ \ \text{and}\ \Lambda(t) =-k\tau+jh$$
for $t\in [-k\tau+jh,-k\tau+(j+1)h)$, it follows that the continuous version of \eqref{eq:RandM} is therefore
\begin{align}
  \label{eq:con_RandM}
  \begin{split}
    \hat{X}_{t}^{n,-k\tau} &= S_n(t+k\tau)P_n \xi +\int_{-k\tau} ^t   S_n\big(t-\Lambda(s)\big) f_n\big(\Lambda(s),   \bar{X}_{s}^{n,-k\tau} \big)\mathrm{d}s\\
    &\quad + 
    \int_{-k\tau} ^t   S_n\big(t-\Lambda(s)\big) g_n\big(\Lambda(s),  \bar{X}_{s}^{n,-k\tau}\big)\mathrm{d}W(s)\\
   &= S(t+k\tau)P_n \xi +\int_{-k\tau} ^t   S\big(t-\Lambda(s)\big) f_n\big(\Lambda(s),   \bar{X}_{s}^{n,-k\tau} \big)\mathrm{d}s\\
   &\quad+ 
    \int_{-k\tau} ^t   S\big(t-\Lambda(s)\big) g_n\big(\Lambda(s), \bar{X}_{s}^{n,-k\tau} \big)\mathrm{d}W(s),
  \end{split}
\end{align}
with differential form
\begin{align}\label{eqn:diffform}
  \begin{split}
     \mathrm{d}\hat{X}_{t}^{n,-k\tau} &= -A  \hat{X}_{t}^{n,-k\tau} +  S\big(t-\Lambda(t)\big) f_n\big(\Lambda(t),   \bar{X}_{t}^{n,-k\tau} \big)\mathrm{d}t\\
     &\quad + 
     S\big(t-\Lambda(t)\big) g_n\big(\Lambda(t),\bar{X}_{t}^{n,-k\tau} \big)\mathrm{d}W(t).
  \end{split}
\end{align}
In Section \ref{sec:numerical}, we show the uniform boundedness of $\hat{X}_{t}^{n,-k\tau}$ by imposing another assumptions on $f$ and $g$:
\begin{assumption}[Dissipative condition]\label{as:dissipative}It holds that
$$2( f(t,u_1) - f(t,u_2),u_1-u_2 )+\|g(t,u_1)-g(t,u_2)\|^2_{\mathcal{L}^2_0}\leq -C_{f,g} \| u_1-u_2\|^2$$ 
 for all $u,u_1, u_2 \in H$ and $t \in [0,\tau)$.
\end{assumption}

\begin{assumption}
  \label{as:lambda2}
$L_{f,g}<2\lambda_1,\ \frac{C_f}{\lambda_1}+\frac{C_g}{\sqrt{\lambda_1}}<1$.
\end{assumption}
We conclude the random periodicity of the spatio-temporal discrete scheme \eqref{eq:con_RandM} in Theorem \ref{thm:main2} and determine a uniform and strong order to approximate $X^{-k\tau}_{\cdot}(\xi)$ from \eqref{eq:con_RandM} in Theorem \ref{thm:error}. Compared to the convergence in SDE cases in \cite{rpsnumerics2017, wu2021}, for the SEE case it is required that the approximation trajectory starting from $L^{2}(\Omega;
  \dot{H}^r)$ with $r\in (0,1)$ rather than an arbitrary starting point in $L^{2}(\Omega; H)$, which guarantees the continuity of the path $X^{-k\tau}_{\cdot}(\xi)$. An interesting observation from it is, the order of convergence directly depends on the space where the initial point lives on, ie, $L^{2}(\Omega;\dot{H}^r)$. As the rate of the convergence from $X^{-k\tau}_{\cdot}(\xi)$ to the random periodic path $X^{*}_{\cdot}$
is dependent of the initial condition, we end up this paper with Corollary \ref{corr:error} which determines a strong but not optimal order for approximating $X^{*}_{\cdot}$. Corollary \ref{corr:error} also implies the best order of convergence can be ever achieved is $1/2-\epsilon$ with arbitrarily small $\epsilon$.

\section{Preliminaries}
In this section we present a few useful mathematical tools for later use.

\begin{prop}\label{prop:semiprop}
Under the condition of the infinitesimal generator $-A$ in Assumption \ref{as:A} for the semigroup $(S(t))_{t \in [0,\infty)}$, the following properties hold:
\begin{enumerate}
    \item For any $\nu \in [0,1]$, there exists a positive constant $C_1(\nu)$ such that
    \begin{equation}\label{eqn:Adiff}
        \|A^{-\nu}(S(t)-\text{Id})\|_{\mathcal{L}(H)}\leq C_1(\nu) t^\nu \quad \text{for }t\geq 0,
    \end{equation}
    where $\text{Id}$ is the identity map from $H$ to $H$.
    In addition,
    \begin{equation}\label{eqn:Ainverse}
    \|A^{-\nu}\|_{\mathcal{L}(H)}\leq \lambda_{1}^{-\nu}.
    \end{equation}
    \item  For any $\mu \geq 0$, there exists a positive constant $C_2(\mu)$ such that
    \begin{equation}\label{eqn:Aestimate}
        \|A^{\mu}S(t)\|_{\mathcal{L}(H)}\leq C_2(\mu) t^{-\mu} \quad \text{for }t> 0.
    \end{equation}
    \item For any $t\geq0$,  $ \|S(t)\|_{\mathcal{L}(H)}\leq e^{-\lambda_{1} t}$.  For the orthogonal projection $P_n$, it holds that
      \begin{equation}\label{eqn:Sdiffestimate}
        \|S(t)(\text{Id}-P_n)\|_{\mathcal{L}(H)}\leq e^{-\lambda_{n+1} t}, \quad \text{for }t\geq 0.
    \end{equation}
\end{enumerate}
\end{prop}
\begin{proof}
The proof for the first two inequalities can be found in \cite{pazy1983}. For the last one, note for any $x\in H$, we have the decomposition $x=\sum_{i=1}^\infty (x,e_i)e_i$. Clearly $S(t)(\text{Id}-P_n)$ is a linear operator from $H$ to $H$. Then the induced norm (indeed we consider its square for convenience) for it is therefore 
\begin{align*}
    \|S(t)(\text{Id}-P_n)\|^2_{\mathcal{L}(H)}&=\sup_{x\in H, \|x\|=1} \|S(t)(\text{Id}-P_n)x\|^2=\sup_{x\in H, \|x\|=1} \sum_{i=n+1}^\infty e^{-2\lambda_i t}(x,e_i)^2 \\
    &\leq e^{-2\lambda_{n+1} t}\sup_{x\in H, \|x\|=1} \sum_{i=1}^\infty (x,e_i)^2 \leq e^{-2\lambda_{n+1} t} \|x\|^2.
\end{align*}
\end{proof}

As one of the main tools, Gamma function is presented:
\begin{align}\label{eqn:gamma}
    \Gamma(\gamma):=\int_0^\infty x^{\gamma-1}e^{-x} \mathrm{d}x<\infty \ \ \text{for\ } \gamma>0.
\end{align}
\section{Existence and uniqueness of random periodic solution}\label{sec:existence}
In the following, we will show the boundedness of the solution to SEE \eqref{eq:SPDE} and characterize its dependence on the initial condition, both of which are crucial ingredients for the existence of random periodic solutions. The proof simply follows Lemma 3.1 and Lemma 3.2 in \cite{wu2021}.  
\begin{lemma}\label{lem:boundedness}
For SEE \eqref{eq:SPDE} with the given initial condition $\xi$ and satisfying Assumption \ref{as:A} to Assumption \ref{as:lambda1}, we have
\begin{equation} \label{eqn:boundedness}
    \sup_{k\in \mathbb{N}}\sup_{t>-k\tau}\mathbb{E}[\|X_{t}^{-k\tau}(\xi)\|^2]<\infty.
\end{equation}
If, in addition, $\xi\in L^{2}(\Omega, \mathcal{F}_{-k\tau}, \mathbb{P};
  \dot{H}^r)$ for some $r\in (0,1)$, then the mild solution $X_{t}^{-k\tau}(\xi)$ introduced in \eqref{eq:mild} is well defined in $ L^{2}(\Omega, \mathcal{F}_{-k\tau},\mathbb{P};
   \dot{H}^r)$ for any $k\in \mathbb{N}$, and $t>-k\tau$.
\end{lemma}
\begin{proof} The fist assertion follows Lemma 3.1 in \cite{wu2021}. It remains to justify the second assertion, by bounding each term of \eqref{eq:mild} in $ L^{2}(\Omega, \mathcal{F}_{-k\tau},\mathbb{P};
  \dot{H}^r)$ with some constant independent of $k$ and $t$.  For the first term on the right hand side of \eqref{eq:mild}, we have   
\begin{align*}
    \mathbb{E}[\|A^{\frac{r}{2}} S(t+k\tau)\xi\|^2] =  \mathbb{E}[\|S(t+k\tau)A ^{\frac{r}{2}}\xi\|^2] \leq \mathbb{E}[\|A^{\frac{r}{2}} \xi\|^2].
\end{align*}
To bound the second term on the right hand side of \eqref{eq:mild}, we apply the linear growth of $f$ in \eqref{eq3:linear_f},  Proposition \ref{prop:semiprop} and \eqref{eqn:boundedness}, and take $\theta=\frac{1}{2}$ as follows:
\begin{align*}
 &\mathbb{E}\big[ \big\|A^{\frac{r}{2}}\int_{-k\tau}^t S(t - s) f(s,X^{-k\tau}_s) \mathrm{d}{s}\big\|^2\big]\\
 &= \mathbb{E}\big[ \big\|\int_{-k\tau}^t A^{\frac{r}{2}}S\big(\theta(t-s)\big)S\big((1-\theta)(t-s)\big) f(s,X^{-k\tau}_s) \mathrm{d}{s}\big\|^2\big]\\
 &\leq 2L_{f,g}^2 \big(1+\sup_{k\in \mathbb{N}}\sup_{s\geq -k\tau}\mathbb{E}[\|X^{-k\tau}_s\|^2]\big)\Big(\int_{-k\tau}^t \|A^{\frac{r}{2}}S\big(\theta(t-s)\big)\|_{\mathcal{L}(H)}\|S\big((1-\theta)(t-s)\big)\| \mathrm{d}{s}\Big)^2\\
 &\leq 2L_{f,g}^2 C_2\big(\frac{r}{2}\big)^2 \big(1+\sup_{k\in \mathbb{N}}\sup_{s\geq -k\tau}\mathbb{E}[\|X^{-k\tau}_s\|^2]\big)\Big(\int_{-k\tau}^t \big(\theta(t-s)\big)^{-\frac{r}{2}} e^{-{\lambda_1}  (1-\theta)(t-s)} \mathrm{d}{s}\Big)^2\\
 &\leq  2L_{f,g}^2 C_2\big(\frac{r}{2}\big)^2 \big(1+\sup_{k\in \mathbb{N}}\sup_{s\geq -k\tau}\mathbb{E}[\|X^{-k\tau}_s\|^2]\big)\lambda_1 ^{r-2}\frac{\Gamma\big(1-\frac{r}{2}\big)^2}{4},
\end{align*}
where we make use of the definition of Gamma function \eqref{eqn:gamma} and $\theta=\frac{1}{2}$ to get
\begin{align*}
    \int_{-k\tau}^t \big(\theta(t-s)\big)^{-\frac{r}{2}} e^{-\lambda_1 (1-\theta)(t-s)} \mathrm{d}{s}=\int_{0}^{t+k\tau} (\theta s)^{-\frac{r}{2}} e^{-\lambda_1 (1-\theta)s} \mathrm{d}{s}\leq \lambda_1 ^{\frac{r}{2}-1}\frac{\Gamma\big(1-\frac{r}{2}\big)}{2}.
\end{align*}
It remains to estimate the last term of \eqref{eq:mild}. To achieve it, we shall apply the It\^o isometry, the linear growth of $g$ in \eqref{eq3:linear_f},  and Proposition \ref{prop:semiprop} together with the technique involving Gamma function above: 
\begin{align*}
 &\mathbb{E}\big[ \big\|A^{\frac{r}{2}}\int_{-k\tau}^t S(t-s) g(s,X^{-k\tau}_s)
  \mathrm{d}{W(s)}\big\|^2\big]\\
  &\leq 2L_{f,g}^2  \big(1+\sup_{k\in \mathbb{N}}\sup_{s\geq -k\tau}\mathbb{E}[\|X^{-k\tau}_s\|^2]\big)\int_{-k\tau}^t \|A^{\frac{r}{2}}S(t-s)\|^2
  \mathrm{d}s\\
  &\leq  2L_{f,g}^2  \big(1+\sup_{k\in \mathbb{N}}\sup_{s\geq -k\tau} \mathbb{E}[\|X^{-k\tau}_s\|^2]\big)(2{\lambda_1})^{r-1}\frac{\Gamma\big(1-r\big)}{2}.
\end{align*}
\end{proof}
\begin{remark}
As in \cite{wu2021}, it suffices to show \eqref{eqn:boundedness} through a weaker condition on $f,g$ than the linear growth \eqref{eq3:linear_f} there exists a constant $\hat{L}_{f,g}$ such that $2(u,f(t,u))+\|g(t,u)\|_{L^2_0}\leq \hat{L}_{f,g}(1+\|u\|^2)$, for $t\in \mathbb{R}$ and $u\in H$.  
\end{remark}

\begin{lemma}\label{lem:stable}
Assume Assumption \ref{as:A} to Assumption \ref{as:lambda1}. Denote by $X_t^{-k\tau}$ and $Y_t^{-k\tau}$ two solutions of SPDE \eqref{eq:SPDE} with different initial values $\xi$ and $\eta$. Then for every $\epsilon>0$, there exists a $t\geq -k\tau$ such that 
\begin{equation}\label{eqn:stable1}
    \mathbb{E}[\|X_{\tilde{t}}^{-k\tau}-Y_{\tilde{t}}^{-k\tau}\|^2]<\epsilon
\end{equation}

whenever $\tilde{t}\geq t$.
\end{lemma}

The existence of the semiflow $u$ for SEE \eqref{eq:SPDE} and its continuity with respect to the initial condition, ie, $u(t,s,\cdot,\omega):H\to H$ being continuous, can be guaranteed by \cite{semiflow2008}. With Lemma \ref{lem:boundedness}, Lemma \ref{lem:stable} and Assumption \ref{as:preserve}, the existence and uniqueness of the random periodic solution to \eqref{eq:SPDE} can be shown following a similar argument in the proof of Theorem 2.4 in \cite{rpsnumerics2017}.

\begin{thm}\label{thm:main1} Under Assumption \ref{as:A} to \ref{as:preserve}, there exists a unique random periodic solution $X^{*}(t,\cdot)\in L^2(\Omega;H)$ such that the solution of \eqref{eq:SPDE} satisfies
\begin{equation}\label{eqn:lim_sde}
    \lim_{k\to \infty}\mathbb{E}[\|X^{-k\tau}_t(\xi)-X^*_t\|^2]=0.
\end{equation}
Moreover, it holds that the mild form of $X^{*}$ given in \eqref{eq:limit} is well defined in $L^2(\Omega;\dot{H}^r)$ for any $r\in (0,1)$. 
\end{thm}
\begin{proof} It remains to show the second assertion. The first conclusion of Theorem \ref{thm:main1} ensures that for any $\epsilon$, there exists a $K(t)\in \mathbb{N}$ such that 
$\mathbb{E}[\|X^{-k\tau}_t(\xi)-X^*_t\|^2]<\epsilon$ for any $k \geq K(t)$.
Then
\begin{align*}
&\limsup_{t \in[0,\pi]}\mathbb{E}[\|X^*_t\|^2]=\limsup_{t \in[0,\pi]}\mathbb{E}[\|X^*_t-X^{-k\tau}_t(\xi)+X^{-k\tau}_t(\xi)\|^2]\\
&\leq \sup_{k\in \mathbb{N}}\sup_{t\in [0,\pi]} 2\mathbb{E}[\|X^{-k\tau}_t(\xi)\|^2]+\limsup_t \lim_{k\geq K(t)}2\mathbb{E}[\|X^{-k\tau}_t(\xi)-X^*_t\|^2]\\
&< \sup_{k\in \mathbb{N}}\sup_{t\in [0,\pi]} 2\mathbb{E}[\|X^{-k\tau}_t(\xi)\|^2]+2\epsilon.
\end{align*}
Because $\epsilon$ is arbitrary, then $\limsup_{t \in[0,\pi]}\mathbb{E}[\|X^*_t\|^2]\leq \sup_{k\in \mathbb{N}}\sup_{t\in [0,\pi]} 2\mathbb{E}[\|X^{-k\tau}_t(\xi)\|^2]$.

Due to the random periodicity of $X^{*}$ and the measure preserving property of $\theta$, it holds that
$$\limsup_{t \in[\pi,2\pi]}\mathbb{E}[\|X^*_t(\cdot)\|^2]=\limsup_{t \in[\pi,2\pi]}\mathbb{E}[\|X^*_{t-\pi}(\theta_\pi \cdot)\|^2]=\limsup_{t \in[\pi,2\pi]}\mathbb{E}[\|X^*_{t-\pi}( \cdot)\|^2]=\limsup_{t \in[0,\pi]}\mathbb{E}[\|X^*_{t}( \cdot)\|^2].$$ Similarly $\limsup_{t \in[-\pi,0]}\mathbb{E}[\|X^*_t\|^2]=\limsup_{t \in[0,\pi]}\mathbb{E}[\|X^*_t\|^2]$. Thus by induction, $\limsup_{t \in \mathbb{R}}\mathbb{E}[\|X^*_t\|^2]<\infty$. Then following the same approach in the proof of Lemma \ref{lem:boundedness}, we can deduce that the mild form of $X^*_t$ is in $L^2(\Omega; \dot{H}^r)$ for any $r\in (0,1)$. 
%As $\text{Dom}(A)\subset \dot{H}^{r}$, then $L^2(\Omega; \bigcap_{r\in (0,1)}\dot{H}^r)$ is not an empty set and at least contains $\text{Dom}(A)$. 
\end{proof}

The second conclusion of Theorem \ref{thm:main1} claims that the $X^*$ lives in an intersection space of $L^2(\Omega;\dot{H}^r)$, which is much smaller than $L^2(\Omega;H)$. Note that the first conclusion of Theorem \ref{thm:main1} shows the convergence is regardless of the initial condition $\xi$, that is,  $X^{-k\tau}_t(\xi)$ will converge to the unique random periodic solution no matter where it starts from. This observation is crucial in that one may choose a starting point with preferred properties, for instance, the continuity shown in Lemma \ref{lem:solcontinuity}.  
\begin{lemma}\label{lem:solcontinuity}
Recall that  for a fixed $h\in (0,1)$, $\Lambda(t) :=-k\tau+jh$
when $t\in (-k\tau+jh,-k\tau+(j+1)h]$. Consider the mild solution $X^{-k\tau}_{\cdot}(\xi)$ of SEE \eqref{eq:SPDE} with given initial condition $\xi\in L^{2}(\Omega, \mathcal{F}_{-k\tau}, \mathbb{P};
  \dot{H}^r)$ for some $r\in (0,1)$ and satisfying Assumption \ref{as:A} to \ref{as:lambda1}. Then for any $\nu_1 \in \big(0,r/2\big]$, there exists a positive constant $C_X$ depending on $r$ and $\nu_1$ such that 
$$\sup_{k\in \mathbb{N}}\sup_{t\geq k\tau}\mathbb{E}[\|X_t^{-k\tau}-X_{\Lambda(t)}^{-k\tau}\|^2]\leq C_X(\nu_1,r) h^{2\nu_1}.$$
\end{lemma}
\begin{proof} One can deduce the following expression from the mild form \eqref{eq:mild}:
\begin{align}\label{eqn:Sdecompose}
\begin{split}
      &X^{-k\tau}_t(\xi)- X^{-k\tau}_{\Lambda(t)}(\xi)\\
  &= \big(S(t-\Lambda(t))-\text{Id}\big)S(\Lambda(t)+k\tau) \xi \\
  &\quad + \int_{\Lambda(t)}^t S(t - s) f(s,X^{-k\tau}_s) \mathrm{d}{s} + \int_{-k\tau}^{\Lambda(t)} \big(S(t-\Lambda(t))-\text{Id}\big)S(\Lambda(t)-s) f(s,X^{-k\tau}_s) \mathrm{d}{s}\\
  &\quad + \int_{\Lambda(t)}^t S(t-s) g(s,X^{-k\tau}_s)
  \mathrm{d}{W(s)}+\int_{-k\tau}^{\Lambda(t)} \big(S(t-\Lambda(t))-\text{Id}\big)S(\Lambda(t)-s) g(s,X^{-k\tau}_s)
  \mathrm{d}{W(s)}.
\end{split}
\end{align}
To get the final assertion, we estimate each term on the right hand in $\mathbb{E}[\|\cdot\|^2]$. For the first term, we have that
\begin{align*}
    &\mathbb{E}\big[\big\|\big(S(t-\Lambda(t))-\text{Id}\big)S(\Lambda(t)+k\tau) \xi\big\|^2\big]\\
    &=\mathbb{E}\big[\big\|A^{-\nu_1}\big(S(t-\Lambda(t))-\text{Id}\big)A^{-(\frac{r}{2}-\nu_1)}S(\Lambda(t)+k\tau) A^{\frac{r}{2}}\xi\big\|^2\big]\\
    &\leq \|A^{-\nu_1}\big(S(t-\Lambda(t))-\text{Id}\big)\|^2_{\mathcal{L}(H)} \|A^{-(\frac{r}{2}-\nu_1)}\|^2_{\mathcal{L}(H)}\|S(\Lambda(t)+k\tau)\|^2_{\mathcal{L}(H)} \mathbb{E}[ \|A^{\frac{r}{2}}\xi\|^2]\\
    &\leq C_1(\nu_1)h^{2\nu_1}\lambda_1^{-(r-2\nu_1)}\mathbb{E}[ \|A^{\frac{r}{2}}\xi\|^2],
\end{align*}
where Proposition \ref{prop:semiprop} is applied for the last line. For the second term of \eqref{eqn:Sdecompose}, by making use of the linear growth condition on $f$ and H\"older's inequality, we can obtain
\begin{align*}
        &\mathbb{E}\big[\big\|\int_{\Lambda(t)}^t S(t - s) f(s,X^{-k\tau}_s) \mathrm{d}{s}\big\|^2\big]\leq2L_{f,g}^2 h^2\big(1+\sup_{k\in \mathbb{N}}\sup_{s\geq -k\tau} \mathbb{E}[\|X^{-k\tau}_s\|^2]\big).
\end{align*}
Similarly for the fourth term of \eqref{eqn:Sdecompose}, through the It\^o isometry we have that
\begin{align*}
        &\mathbb{E}\big[\big\|\int_{\Lambda(t)}^t S(t-s) g(s,X^{-k\tau}_s)
  \mathrm{d}{W(s)}\big\|^2\big]\\
  &=\int_{\Lambda(t)}^t \mathbb{E}\big[\|S(t-s) g(s,X^{-k\tau}_s)\|_{\mathcal{L}^2_0}\big]
  \mathrm{d}{s}\\
  &\leq 2L_{f,g}^2 h\big(1+\sup_{k\in \mathbb{N}}\sup_{s\geq -k\tau} \mathbb{E}[\|X^{-k\tau}_s\|^2]\big).
\end{align*}
For the third term of \eqref{eqn:Sdecompose}, applying Assumption \ref{as:A}, Proposition \ref{prop:semiprop}, and defining $\theta=1/2$ yield the following estimate
\begin{align}\label{eqn:Sdiffintegralestimate}
\begin{split}
  &\mathbb{E}\big[\big\|\int_{-k\tau}^{\Lambda(t)} \big(S(t-\Lambda(t))-\text{Id}\big)S(\Lambda(t)-s) f(s,X^{-k\tau}_s) \mathrm{d}{s}\big\|^2\big]\\
  &=\mathbb{E}\big[\big\|\int_{-k\tau}^{\Lambda(t)} A^{-\nu_1}\big(S(t-\Lambda(t))-\text{Id}\big)A^{\nu_1}S\big(\Lambda(t)-s\big) f(s,X^{-k\tau}_s) \mathrm{d}{s}\big\|^2\big]\\
  &\leq C_1(\nu_1)^2h^{2\nu_1}\int_{-k\tau}^{\Lambda(t)}\|A^{\nu_1}S\big(\Lambda(t)-s\big)\|_{\mathcal{L}(H)}\mathrm{d}{s}\int_{-k\tau}^{\Lambda(t)}\|A^{\nu_1}S\big(\Lambda(t)-s\big)\|_{\mathcal{L}(H)}\mathbb{E}[\| f(s,X^{-k\tau}_s)\|^2]\mathrm{d}{s}\\
  &\leq 2L_{f,g}^2 \big(1+\sup_{k\in \mathbb{N}}\sup_{s\geq -k\tau}\mathbb{E}[\|X^{-k\tau}_s\|^2]\big) C_1(\nu_1)^2h^{2\nu_1}\Big(\int_{0}^{\Lambda(t)+k\tau}\|A^{\nu_1}S(\theta s)S\big((1-\theta)s\big)\|_{\mathcal{L}(H)}\mathrm{d}{s}\Big)^2\\
  &\leq  2L_{f,g}^2 \big(1+\sup_{k\in \mathbb{N}}\sup_{s\geq -k\tau}\mathbb{E}[\|X^{-k\tau}_s\|^2]\big) C_1(\nu_1)^2h^{2\nu_1}C_2(\nu_1)^2\Big(\int_{0}^{\Lambda(t)+k\tau}(\theta s)^{-\nu_1}e^{-\lambda_1 (1-\theta)s}\mathrm{d}{s}\Big)^2\\ 
  &\leq 2L_{f,g}^2 \big(1+\sup_{k\in \mathbb{N}}\sup_{s\geq -k\tau}\mathbb{E}[\|X^{-k\tau}_s\|^2]\big) C_1(\nu_1)^2h^{2\nu_1}C_2(\nu_1)^2\frac{\lambda_1^{2(\nu_1-1)} \Gamma(1-\nu_1)^2}{4}, 
  \end{split}
\end{align}
where we change variable to deduce the integral in the fourth line and apply the Gamma function \eqref{eqn:gamma} to get the last line.

For the last term of \eqref{eqn:Sdecompose}, using the It\^o isometry, the linear growth of $g$ in \eqref{eq3:linear_f} and the definition of the Gamma function we have that
\begin{align}\label{eqn:SdiffWintegralestimate}
\begin{split}
          &\mathbb{E}\big[\big\|\int_{-k\tau}^{\Lambda(t)} \big(S(t-\Lambda(t))-\text{Id}\big)S(\Lambda(t)-s) g(s,X^{-k\tau}_s)
  \mathrm{d}{W}(s)\big\|^2\big]\\
  &=\int_{-k\tau}^{\Lambda(t)} \mathbb{E}\big[\|A^{-\nu_1}\big(S(t-\Lambda(t))-\text{Id}\big)A^{\nu_1}S(\Lambda(t)-s) g(s,X^{-k\tau}_s)\|^2_{\mathcal{L}^2_0}\big]
  \mathrm{d}{s}\\
  & \leq 2L_{f,g}^2\big(1+\sup_{k\in \mathbb{N}}\sup_{s\geq -k\tau}\mathbb{E}[\|X^{-k\tau}_s\|^2]\big)  C_1(\nu_1)^2h^{2\nu_1} \int_{0}^{\Lambda(t)+k\tau} \|A^{\nu_1}S(\theta s)S\big((1-\theta) s\big)\|^2_{\mathcal{L}(H)}\mathrm{d}{s}\\
  &\leq 2L_{f,g}^2\big(1+\sup_{k\in \mathbb{N}}\sup_{s\geq -k\tau}\mathbb{E}[\|X^{-k\tau}_s\|^2]\big)  C_1(\nu_1)^2h^{2\nu_1}
C_2(\nu_1)^2\frac{2(2{\lambda_1}^{2\nu_1-1} \Gamma(1-2\nu_1)^2}{2}.  
\end{split}
\end{align}
\end{proof}
One will see that the continuity of the true solution in Lemma \ref{lem:solcontinuity} plays an important role in later analysis.
\section{The random periodic solution of the Galerkin numerical approximation}\label{sec:numerical}
This section is devoted to the existence and uniqueness of the random periodic solution for the Galerkin-type spatio-temporal discretization defined in \eqref{eq:con_RandM}, and its convergence to the random periodic solution of our underlying SPDE \eqref{eq:SPDE}.

\begin{lemma}\label{lem:onestep_estimate}
Under Assumption \ref{as:A} to Assumption \ref{as:fg}, for the continuous version of the numerical scheme defined in \eqref{eq:con_RandM} with stepsize $h\in (0,1)$, it holds that 
\begin{align}\label{eqn:onestep_estimate}
    \mathbb{E}[\|\hat{X}_{t}^{n,-k\tau}-\bar{X}_{t}^{n,-k\tau}\|^2]\leq C_n h\big(1+\mathbb{E}\big[\big\|\bar{X}_{\Lambda(t)}^{n,-k\tau}\big\|^2\big]\big),
\end{align}
where $C_n=3(\lambda_n^2+4L_{f,g}^2)$.
\end{lemma}
\begin{proof}
From \eqref{eqn:diffform} we get that
\begin{align}\label{eqn:onestep_full}
\begin{split}
    &\hat{X}_{t}^{n,-k\tau} -\bar{X}_{t}^{n,-k\tau} \\
   &= \big(S\big(t-\Lambda(t)\big)-\text{Id}\big) \bar{X}_{t}^{n,-k\tau}  +\int_{\Lambda(t)} ^t   S\big(t-\Lambda(s)\big) f_n\big(\Lambda(s),   \bar{X}_{s}^{n,-k\tau} \big)\mathrm{d}s\\
   &\quad + 
    \int_{\Lambda(t)} ^t   S\big(t-\Lambda(s)\big) g_n\big(\Lambda(s),   \bar{X}_{s}^{n,-k\tau} \big)\mathrm{d}W(s).
\end{split}
\end{align}
For the first term on the right hand side, we have that
\begin{align}\label{eqn:operatordiff}
\begin{split}
    &\mathbb{E}[\|\big(S\big(t-\Lambda(t)\big)-\text{Id}\big) \bar{X}_{t}^{n,-k\tau} \|^2]\\
    &=  \mathbb{E}\Big[\Big\|\sum_{i=1}^n \big(e^{-\lambda_i(t-\Lambda(t))}-1\big)\big(e_i, \bar{X}_{t}^{n,-k\tau}\big)e_i  \Big\|^2\Big]\\
    &\leq \big(e^{-\lambda_n(t-\Lambda(t))}-1\big)^2\mathbb{E}\big[\|\bar{X}_{t}^{n,-k\tau}\|^2\big]\leq \lambda_n^2 h^2 \mathbb{E}\big[\big\|\bar{X}_{\Lambda(t)}^{n,-k\tau}\big\|^2\big],
\end{split}
\end{align}
where we use the fact $(1-e^{-a})\leq a$ for $a> 0$ to derive the last inequality. 

For the second term on the right hand side of \eqref{eqn:onestep_full}, we have that
\begin{align*}
    &\mathbb{E}\Big[\Big\|\int_{\Lambda(t)} ^t   S\big(t-\Lambda(s)\big) f_n\big(\Lambda(s),   \bar{X}_{s}^{n,-k\tau} \big)\mathrm{d}s\Big\|^2\Big]\\
    &\leq  \int_{\Lambda(t)} ^t   \|S\big(t-\Lambda(s)\big)\|^2_{\mathcal{L}(H)}\mathrm{d}s \int_{\Lambda(t)} ^t    \mathbb{E}\big[\big\|f_n\big(\Lambda(s),   \bar{X}_{s}^{n,-k\tau} \big)\big\|^2\big]\mathrm{d}s\\
    &\leq 2h^2L_{f,g}^2 \big(1+ \mathbb{E}\big[\big\|\bar{X}_{\Lambda(t)}^{n,-k\tau}\big\|^2\big]\big),
\end{align*}
where we apply the H\"older inequality to deduce the second line and make use of the linear growth of $f$ to get the last line.

For the last term on the right hand side of \eqref{eqn:onestep_full},  through the It\^o isometry, Assumption \ref{as:A} and the linear growth of $g$ we have that
\begin{align*}
    &\mathbb{E}\Big[\Big\|\int_{\Lambda(t)} ^t   S\big(t-\Lambda(s)\big) g_n\big(\Lambda(s),\bar{X}_{s}^{n,-k\tau}\big)\mathrm{d}W(s)\Big\|^2\Big]\\
    &= \int_{\Lambda(t)} ^t    \mathbb{E}\big[\|S\big(t-\Lambda(s)\big)g_n\big(\Lambda(s),\bar{X}_{s}^{n,-k\tau}\big) \big\|^2_{\mathcal{L}^2_0}\big]\mathrm{d}s\\
    &\leq 2h L_{f,g}^2\big(1+ \mathbb{E}\big[\big\|\bar{X}_{\Lambda(t)}^{n,-k\tau}\big\|^2\big]\big).
\end{align*}

\end{proof}

\begin{lemma}\label{lem:boundednessG}
Under Assumption \ref{as:A} to Assumption \ref{as:fg} and Assumption \ref{as:dissipative} to Assumption \ref{as:lambda2}, let $X^{-k\tau}_{\cdot}$ be the solution of SEE \eqref{eq:SPDE} with the initial condition $\xi$ and $\hat{X}^{n,-k\tau}_{\cdot}$ from \eqref{eq:con_RandM} be its numerical simulation with the stepsize $h$ satisfying 
\begin{align}\label{eqn:hchoice}
    \big(5L_{f,g} \sqrt{h\lambda_n}(1+C_n h)+2C_f\sqrt{C_n}\big) \sqrt{h}\leq 2C_{f,g}.
\end{align}
Then it holds that
\begin{equation} 
    \sup_{k\in \mathbb{N}}\sup_{t>-k\tau}\mathbb{E}[\|\hat{X}_{t}^{n,-k\tau}(\xi)\|^2]\leq C_{\xi}^2+\frac{4(C_{f,g}+L_{f,g}+L^2_{f,g})}{2\lambda_1-L_{f,g}}.
\end{equation}
If, in addition, $\xi\in L^{2}(\Omega, \mathcal{F}_{-k\tau}, \mathbb{P};
  \dot{H}^r)$ for $r\in (0,1)$, the numerical solution introduced in \eqref{eq:con_RandM} is well defined in $ L^{2}(\Omega, \mathcal{F}_{-k\tau},\mathbb{P};
   \dot{H}^r)$ for any $k\in \mathbb{N}$, and $t>-k\tau$.
\end{lemma}
\begin{proof} 
Applying the It\^o formula to $e^{2\lambda t}\|\hat{X}_{t}^{n,-k\tau}(\xi)\|^2$, where we consider the differential form \eqref{eqn:diffform}, and taking the expectation yield
\begin{align}\label{eqn:fiveterms}
\begin{split}
     e^{2\lambda t}\mathbb{E}[\|\hat{X}_{t}^{n,-k\tau}(\xi)\|^2]=& e^{-2\lambda k\tau}\mathbb{E}[\|\xi\|^2]+2\lambda \int_{-k\tau}^t e^{2\lambda s}\mathbb{E}[\|\hat{X}_{s}^{n,-k\tau}\|^2]\mathrm{d}s\\
 &-2\int_{-k\tau}^te^{2\lambda s}\mathbb{E}\big( \hat{X}_{s}^{n,-k\tau}, A\hat{X}_{s}^{n,-k\tau}\big)\mathrm{d}s\\
 &+2\int_{-k\tau}^te^{2\lambda s}\mathbb{E}\big( \hat{X}_{s}^{n,-k\tau}, S\big(s-\Lambda(s)\big) f_n(\Lambda(s),   \bar{X}_{s}^{n,-k\tau} )\big) \mathrm{d}s\\
 &+\int_{-k\tau}^t e^{2\lambda s}\mathbb{E}\big[\big\|S\big(s-\Lambda(s)\big) g_n\big(\Lambda(s),\bar{X}_{s}^{n,-k\tau} \big)\big\|^2_{\mathcal{L}^2_0}\big]\mathrm{d}s.
\end{split}
\end{align}
Note that the inner product in the last second term can be further divided into several inner products as follows:
\begin{align}
    \begin{split}
&\big( \hat{X}_{s}^{n,-k\tau}, S\big(s-\Lambda(s)\big) f_n(\Lambda(s),   \bar{X}_{s}^{n,-k\tau} )\big)   \\   
 &= \big( \hat{X}_{s}^{n,-k\tau}, \big(S\big(s-\Lambda(s)\big) -\text{Id}\big) f_n(\Lambda(s),   \hat{X}_{s}^{n,-k\tau} )\big) \\
 &\quad+ \big( \hat{X}_{s}^{n,-k\tau}-\bar{X}_{s}^{n,-k\tau}, f_n(\Lambda(s),   \bar{X}_{s}^{n,-k\tau} )-f_n(\Lambda(s),   0)\big) \\
  &\quad+ \big( \bar{X}_{s}^{n,-k\tau}, f_n(\Lambda(s),   \bar{X}_{s}^{n,-k\tau} )-f_n(\Lambda(s),   0 )\big) \\
  &\quad+ \big( \hat{X}_{s}^{n,-k\tau}, f_n(\Lambda(s),  0)\big)=:\sum_{i=1}^4 V_i.   
    \end{split}
\end{align}
For $V_1$, we have that
\begin{align*}
    2\int_{-k\tau}^te^{2\lambda s}\mathbb{E}[V_1] \mathrm{d}s&= 2\int_{-k\tau}^te^{2\lambda s}\mathbb{E}\big( \hat{X}_{s}^{n,-k\tau}, \big(S\big(s-\Lambda(s)\big) -\text{Id}\big) f_n(\Lambda(s),   \hat{X}_{s}^{n,-k\tau} )\big)  \mathrm{d}s\\
    &\leq 2L_{f,g}\lambda_n h\int_{-k\tau}^te^{2\lambda s}\mathbb{E}\big[\|\hat{X}_{s}^{n,-k\tau}\|(\|\hat{X}_{s}^{n,-k\tau}\|+1)\big]\mathrm{d}s\\
    &\leq L_{f,g}\lambda_n h\int_{-k\tau}^te^{2\lambda s}\mathrm{d}s+\frac{5}{2}\hat{C}_f\lambda_n h\int_{-k\tau}^te^{2\lambda s}\mathbb{E}\big[\|\hat{X}_{s}^{n,-k\tau}\|^2\big]\mathrm{d}s\\
    &\leq \frac{L_{f,g}\lambda_n h(1+5C_n h)}{2\lambda}(e^{2\lambda t}-e^{-2\lambda k\tau})\\
    &\quad+5L_{f,g} \lambda_n h(1+C_n h)\int_{-k\tau}^te^{2\lambda s}\mathbb{E}\big[\|\bar{X}_{s}^{n,-k\tau}\|^2\big]\mathrm{d}s,
\end{align*}
where to deduce the second line we make use of linear growth of $f$ and a similar estimate of \eqref{eqn:operatordiff}, and to bound the term of $\mathbb{E}\big[\|\hat{X}_{s}^{n,-k\tau}\|^2\big]$ in terms of $\mathbb{E}\big[\|\bar{X}_{s}^{n,-k\tau}\|^2\big]$ to get the last line we make use of the following fact from Lemma \ref{lem:onestep_estimate}:
\begin{align}\label{eqn:usfulest1}
    \begin{split}
    \mathbb{E}\big[\|\hat{X}_{s}^{n,-k\tau}\|^2\big]&\leq 2\mathbb{E}\big[\|\bar{X}_{s}^{n,-k\tau}\|^2\big]+2\mathbb{E}\big[\|\hat{X}_{s}^{n,-k\tau}-\bar{X}_{s}^{n,-k\tau}\|^2\big]\\
    &\leq 2(1+C_nh)\mathbb{E}\big[\|\bar{X}_{s}^{n,-k\tau}\|^2\big]+2C_nh.
    \end{split}
\end{align}

For $V_2$, we have that by the Lipchitz condition of $f$, the H\"older inequality and Lemma \ref{lem:onestep_estimate} 
\begin{align*}
    2\int_{-k\tau}^te^{2\lambda s}\mathbb{E}[V_2] \mathrm{d}s&= 2\int_{-k\tau}^te^{2\lambda s}\mathbb{E} \big( \hat{X}_{s}^{n,-k\tau}-\bar{X}_{s}^{n,-k\tau}, f_n(\Lambda(s),   \bar{X}_{s}^{n,-k\tau} )-f_n(\Lambda(s),   0)\big) \mathrm{d}s\\
    &\leq 2C_f\int_{-k\tau}^te^{2\lambda s}\mathbb{E}\big[\|\hat{X}_{s}^{n,-k\tau}-\bar{X}_{s}^{n,-k\tau}\|\|\bar{X}_{s}^{n,-k\tau}\|\big]\mathrm{d}s\\
    &\leq 2C_f\int_{-k\tau}^te^{2\lambda s}\sqrt{\mathbb{E}\big[\|\hat{X}_{s}^{n,-k\tau}-\bar{X}_{s}^{n,-k\tau}\|^2\big]\mathbb{E}\big[\|\bar{X}_{s}^{n,-k\tau}\|^2\big]}\mathrm{d}s\\
    &\leq 2C_f\sqrt{C_n} \sqrt{h}\int_{-k\tau}^te^{2\lambda s}\sqrt{\big(\mathbb{E}\big[\|\bar{X}_{s}^{n,-k\tau}\|^2\big]+1\big)\mathbb{E}\big[\|\bar{X}_{s}^{n,-k\tau}\|^2\big]}\mathrm{d}s\\
    &\leq \frac{2C_f\sqrt{C_n} \sqrt{h}}{2\lambda}(e^{2\lambda t}-e^{-2\lambda k\tau})+2C_f\sqrt{C_n} \sqrt{h}\int_{-k\tau}^te^{2\lambda s}\mathbb{E}\big[\|\bar{X}_{s}^{n,-k\tau}\|^2\big]\mathrm{d}s,
\end{align*}
where we use the fact $\sqrt{(a^2+1)a^2}\leq a^2+1$ to deduce the last line.

For $V_3$, together with the last term in \eqref{eqn:fiveterms}, we are able to make use of dissipative condition in Assumtion \ref{as:dissipative} such that 
\begin{align*}
    & 2\int_{-k\tau}^te^{2\lambda s}\mathbb{E}[V_3] \mathrm{d}s+ \int_{-k\tau}^t e^{2\lambda s}\mathbb{E}\big[\big\|S\big(s-\Lambda(s)\big) g_n\big(\Lambda(s),\bar{X}_{s}^{n,-k\tau} \big)\big\|^2_{\mathcal{L}^2_0}\big]\mathrm{d}s \\
    &\leq  2\int_{-k\tau}^te^{2\lambda s}\mathbb{E}\big[\big( \bar{X}_{s}^{n,-k\tau}, f_n(\Lambda(s),   \bar{X}_{s}^{n,-k\tau} )-f_n(\Lambda(s),   0 )\big)\\
    &\quad \quad \quad + \big\| g_n\big(\Lambda(s),\bar{X}_{s}^{n,-k\tau} \big)-g_n\big(\Lambda(s),0\big)\big\|^2_{\mathcal{L}^2_0} \big]\mathrm{d}s + \int_{-k\tau}^te^{2\lambda s}\| g_n\big(\Lambda(s),0 \big)\big\|^2_{\mathcal{L}^2_0} \big]\mathrm{d}s\\
    &\leq -2C_{f,g}\int_{-k\tau}^te^{2\lambda s}\mathbb{E}\big[\|\bar{X}_{s}^{n,-k\tau}\|^2\big]\mathrm{d}s+\frac{L_{f,g}^2}{\lambda}(e^{2\lambda t}-e^{-2\lambda k\tau}),
\end{align*}
where we also apply linear growth of $g$ in \eqref{eq3:linear_f} to deduce the last line.

%For $V_3$, we have that by the one-sided Lipchitz condition of $f$
%\begin{align*}
%    2\int_{-k\tau}^te^{2\lambda s}\mathbb{E}[V_3] \mathrm{d}s&= 2\int_{-k\tau}^te^{2\lambda s}\mathbb{E}\big( \bar{X}_{s}^{n,-k\tau}, f_n(\Lambda(s),   \bar{X}_{s}^{n,-k\tau} )-f_n(\Lambda(s),   0 )\big)\mathrm{d}s\\
%    &\leq -2C_{f,g}\int_{-k\tau}^te^{2\lambda s}\mathbb{E}\big[\|\bar{X}_{s}^{n,-k\tau}\|^2\big]\mathrm{d}s.
%\end{align*}

For $V_4$, we have that by the linear growth of $f$
\begin{align*}
   & 2\int_{-k\tau}^te^{2\lambda s}\mathbb{E}[V_4] \mathrm{d}s= 2\int_{-k\tau}^te^{2\lambda s}\mathbb{E}\big( \hat{X}_{s}^{n,-k\tau}, f_n(\Lambda(s),  0)\big)\mathrm{d}s\\
    &\leq 2L_{f,g}\int_{-k\tau}^te^{2\lambda s}\mathbb{E}\big[\|\hat{X}_{s}^{n,-k\tau}\|\big]\mathrm{d}s\\
    &\leq \frac{L_{f,g}}{\lambda}(e^{2\lambda t}-e^{-2\lambda k\tau}) +L_{f,g}\int_{-k\tau}^te^{2\lambda s}\mathbb{E}\big[\|\hat{X}_{s}^{n,-k\tau}\|^2\big]\mathrm{d}s
\end{align*}
%For the last term of \eqref{eqn:fiveterms}, we have that from the linear growth of $g$, Assumption \ref{as:fg} and Lemma \ref{lem:onestep_estimate}
%\begin{align*}
%    &\int_{-k\tau}^t e^{2\lambda s}\mathbb{E}\big[\big\|S\big(s-\Lambda(s)\big) g_n\big(\Lambda(s),\bar{X}_{s}^{n,-k\tau}\big)\big\|^2_{\mathcal{L}^2_0}\big]\mathrm{d}s\\
%    &\leq 2\int_{-k\tau}^t e^{2\lambda s}\mathbb{E}\big[\big\|S\big(s-\Lambda(s)\big) g_n\big(\Lambda(s),\bar{X}_{s}^{n,-k\tau}\big)-g_n\big(\Lambda(s),\hat{X}_{s}^{n,-k\tau}\big)\big\|^2_{\mathcal{L}^2_0}\big]\mathrm{d}s\\
%    &\quad +2\int_{-k\tau}^t e^{2\lambda s}\mathbb{E}\big[\big\|S\big(s-\Lambda(s)\big) g_n\big(\Lambda(s),\hat{X}_{s}^{n,-k\tau}\big)\big\|^2_{\mathcal{L}^2_0}\big]\mathrm{d}s\\
%   & \leq 2C_g^2\int_{-k\tau}^t  \mathbb{E}[\|\hat{X}_{s}^{n,-k\tau}-\bar{X}_{s}^{n,-k\tau}\|^2]\mathrm{d}s+\frac{2L_{f,g}^2}{\lambda}(e^{2\lambda t}-e^{-2\lambda k\tau})+2L_{f,g}^2\int_{-k\tau}^te^{2\lambda s}\mathbb{E}\big[\|\hat{X}_{s}^{n,-k\tau}\|^2\big]\mathrm{d}s\\
%   &\leq 2C_g^2 C_n h \int_{-k\tau}^t\mathbb{E}[\|\bar{X}_{s}^{n,-k\tau}\|^2]\mathrm{d}s+\frac{2L_{f,g}^2+2C_g^2 C_n h}{\lambda}(e^{2\lambda t}-e^{-2\lambda k\tau})+2L_{f,g}^2\int_{-k\tau}^te^{2\lambda s}\mathbb{E}\big[\|\hat{X}_{s}^{n,-k\tau}\|^2\big]\mathrm{d}s.
%\end{align*}

Under Assumption \ref{as:lambda2}, take $\lambda= \lambda_1-L_{f,g}/2$. In summary, 
\begin{align*}
     &e^{2\lambda t}\mathbb{E}[\|\hat{X}_{t}^{n,-k\tau}(\xi)\|^2]\\
     &\leq  e^{-2\lambda k\tau}\mathbb{E}[\|\xi\|^2]+(2\lambda+L_{f,g}-2\lambda_1) \int_{-k\tau}^t e^{2\lambda s}\mathbb{E}[\|\hat{X}_{s}^{n,-k\tau}\|^2]\mathrm{d}s\\
 &\quad -\big(2C_{f,g}-(5L_{f,g} \sqrt{h\lambda_n}(1+C_n h)+2C_f\sqrt{C_n}) \sqrt{h}\big)\int_{-k\tau}^te^{2\lambda s}\mathbb{E}\big[\|\bar{X}_{s}^{n,-k\tau}\|^2\big]\mathrm{d}s\\
 &\quad +\frac{L_{f,g}\sqrt{\lambda_n h}(1+5C_n h)+2C_f\sqrt{C_n h}+2L_{f,g}^2+2L_{f,g}}{2\lambda}(e^{2\lambda t}-e^{-2\lambda k\tau})\\
 &\leq e^{-2\lambda k\tau}\mathbb{E}[\|\xi\|^2]+ \frac{4(C_{f,g}+L_{f,g}+L^2_{f,g})}{2\lambda_1-L_{f,g}}(e^{2\lambda t}-e^{-2\lambda k\tau}),
\end{align*}
where, to deduce the last line, we make use of the choice for $h$ in \eqref{eqn:hchoice}. This leads to
$$\mathbb{E}[\|\hat{X}_{t}^{n,-k\tau}(\xi)\|^2]\leq \mathbb{E}[\|\xi\|^2]+\frac{4(C_{f,g}+L_{f,g}+L^2_{f,g})}{2\lambda_1-L_{f,g}}(1-e^{-2\lambda (t+k\tau)})\leq C_{\xi}^2+\frac{4(C_{f,g}+L_{f,g}+L^2_{f,g})}{2\lambda_1-L_{f,g}}.$$
The second assertion follows the proof of Lemma \ref{lem:boundedness}.
\end{proof}
\begin{lemma}\label{lem:onesteperror_estimate}
Under Assumption \ref{as:A} to Assumption \ref{as:fg}, denote by $\hat{X}_t^{n,-k\tau}$ and $\hat{Y}_t^{n,-k\tau}$ two Galerkin numerical approximations from \eqref{eq:con_RandM} of SEE \eqref{eq:SPDE} with the same stepsize $h\in (0,1)$ but different initial values $\xi$ and $\eta$. Define $\hat{E}_t^{n,-k\tau}:=\hat{X}_t^{n,-k\tau}-\hat{Y}_t^{n,-k\tau}$ and similarly $\bar{E}_t^{n,-k\tau}:=\bar{X}_t^{n,-k\tau}-\bar{Y}_t^{n,-k\tau}$. Then 
\begin{align}\label{eqn:onesteperror_estimate}
    \mathbb{E}[\|\hat{E}_{t}^{n,-k\tau}-\bar{E}_{t}^{n,-k\tau}\|^2]\leq c_n h\mathbb{E}[\big\|\bar{E}_{\Lambda(t)}^{n,-k\tau}\big\|^2],
\end{align}
where $c_n=3(\lambda^2_n+C_f^2+C_g^2)$.
\end{lemma}
\begin{proof}
From \eqref{eqn:diffform}, we have that
\begin{align}\label{eqn:disnum_diffform}
\begin{split}
     \mathrm{d}\hat{E}_t^{n,-k\tau} &= -A  \hat{E}_t^{n,-k\tau} +  S\big(t-\Lambda(t)\big) \big(f_n(\Lambda(t),   \bar{X}_{t}^{n,-k\tau} )-f_n(\Lambda(t),   \bar{Y}_{t}^{n,-k\tau} )\big)\mathrm{d}t\\
     &+ S\big(t-\Lambda(t)\big) \big(g_n(\Lambda(t),   \bar{X}_{t}^{n,-k\tau} )-g_n(\Lambda(t),   \bar{Y}_{t}^{n,-k\tau} )\big)\mathrm{d}W(t).
\end{split}
\end{align}
The rest of the proof is similar to the proof of Lemma \ref{lem:onestep_estimate}.
\end{proof}
\begin{lemma}\label{lem:numsoldependence} Under the same assumptions as Lemma \ref{lem:onesteperror_estimate} and Assumption \ref{as:dissipative}.
Denote by $\hat{X}_t^{n,-k\tau}$ and $\hat{Y}_t^{n,-k\tau}$ two approximations of SEE \eqref{eq:SPDE} with different initial values $\xi$ and $\eta$ under the same stepsize $2C_f(\sqrt{h}\lambda_n+\sqrt{c_n})\sqrt{h}\leq C_{f,g}$. Then 
$$\mathbb{E}[\|\hat{X}_t^{n,-k\tau}-\hat{Y}_t^{n,-k\tau}\|^2]\leq e^{-2\lambda_1(t+k\tau)}\mathbb{E}[\|\xi-\eta\|^2].$$
\end{lemma}
\begin{proof}
Similar as the proof of Lemma \ref{lem:boundednessG}, we apply the It\^o formula to $e^{2\lambda_1 t}\|\hat{E}_t^{n,-k\tau}\|^2$, take its expectation, make use of the It\^o isometry and get
\begin{align}\label{eqn:numerror_full}
    \begin{split}
     e^{2\lambda_1 t}\mathbb{E}[\|\hat{E}_t^{n,-k\tau}\|^2]&=e^{-2\lambda_1 k\tau}\mathbb{E}[\|\xi-\eta\|^2]+2\lambda_1 \int_{-k\tau}^t e^{2\lambda_1 s}\mathbb{E}[\|\hat{E}_{s}^{n,-k\tau}\|^2]\mathrm{d}s\\
 &-2\int_{-k\tau}^te^{2\lambda_1 s}\mathbb{E}\big( \hat{E}_{s}^{n,-k\tau}, A\hat{E}_{s}^{n,-k\tau}\big)\mathrm{d}s\\
     &+2\int_{-k\tau}^te^{2\lambda_1 s}\mathbb{E}\Big(\hat{E}_s^{n,-k\tau}, S\big(s-\Lambda(s)\big)\big(f(\Lambda(s),\bar{X}_{s}^{n,-k\tau})-f(\Lambda(s),\bar{Y}_{s}^{n,-k\tau})\big)\Big)\mathrm{d}s\\
     &+\int_{-k\tau}^t e^{2\lambda_1 s}\mathbb{E}\big[\|S\big(s-\Lambda(s)\big)\big(g(\Lambda(s),\bar{X}_{s}^{n,-k\tau})-g(\Lambda(s),\bar{Y}_{s}^{n,-k\tau})\big)\|^2_{\mathcal{L}^2_0}\big]\mathrm{d}s.
\end{split}
\end{align}
In order to make use of the dissipative condition in Assumption \ref{as:dissipative}, we further decompose the following term into three terms
\begin{align*}
    &\Big(\hat{E}_s^{n,-k\tau}, S\big(s-\Lambda(s)\big)\big(f(\Lambda(s),\bar{X}_{s}^{n,-k\tau})-f(\Lambda(s),\bar{Y}_{s}^{n,-k\tau})\big)\Big)\\
    &=\big(\bar{E}_s^{n,-k\tau}, f(\Lambda(s),\bar{X}_{s}^{n,-k\tau})-f(\Lambda(s),\bar{Y}_{s}^{n,-k\tau})\big)\\
    &\quad +\Big(\bar{E}_s^{n,-k\tau}, \big(S\big(s-\Lambda(s)\big)-\text{Id}\big)\big(f(\Lambda(s),\bar{X}_{s}^{n,-k\tau})-f(\Lambda(s),\bar{Y}_{s}^{n,-k\tau})\big)\Big)\\
    &\quad+ \Big(\hat{E}_s^{n,-k\tau}-\bar{E}_s^{n,-k\tau}, S\big(s-\Lambda(s)\big)\big(f(\Lambda(s),\bar{X}_{s}^{n,-k\tau})-f(\Lambda(s),\bar{Y}_{s}^{n,-k\tau})\big)\Big)\\
    &=:U_1+U_2+U_3.
\end{align*}

Substituting the right hand side into Eqn. \eqref{eqn:numerror_full} and applying the dissipative condition  in Assumption \ref{as:dissipative} give that
\begin{align*}
    &e^{2\lambda_1 t}\mathbb{E}[\|\hat{E}_t^{n,-k\tau}\|^2]\leq  e^{-2\lambda_1 k\tau}\mathbb{E}[\|\xi-\eta\|^2]+2(\lambda_1-\lambda_1)\int_{-k\tau}^te^{2\lambda_1 s}\mathbb{E}[\|\hat{E}_{s}^{n,-k\tau}\|^2]\mathrm{d}s\\
    &-C_{f,g}\int_{-k\tau}^te^{2\lambda_1 s}\mathbb{E}[\|\bar{E}_{s}^{n,-k\tau}\|^2]\mathrm{d}s+2\int_{-k\tau}^te^{2\lambda_1 s}\mathbb{E}[U_2]\mathrm{d}s+2\int_{-k\tau}^te^{2\lambda_1 s}\mathbb{E}[U_3]\mathrm{d}s.
\end{align*}

For the term involving $U_2$, we have that
\begin{align*}
   2\int_{-k\tau}^te^{2\lambda_1 s}\mathbb{E}[U_2]\mathrm{d}s\leq 2C_f \lambda_n h  \int_{-k\tau}^te^{2\lambda_1 s}\mathbb{E}[\|\hat{E}_{s}^{n,-k\tau}\|^2]\mathrm{d}s,
\end{align*}
where we bound $\|S\big(s-\Lambda(s)\big)-\text{Id}\big) \cdot\|\leq \lambda_n h \|\cdot\|$ as we deduce the bound in \eqref{eqn:operatordiff} of Lemma \ref{lem:onestep_estimate} .

For the term involving $U_3$, we have that
\begin{align*}
    2\int_{-k\tau}^te^{2\lambda_1 s}\mathbb{E}[U_3]\mathrm{d}s &\leq 2C_f \int_{-k\tau}^t e^{2\lambda_1 s}\mathbb{E}[\|\hat{E}_s^{n,-k\tau}-\bar{E}_s^{n,-k\tau}\|\|\bar{E}_s^{n,-k\tau}\|]\mathrm{d}s\\
    &\leq 2C_f\sqrt{c_n}\sqrt{h} \int_{-k\tau}^t e^{2\lambda_1 s}\mathbb{E}[\|\bar{E}_s^{n,-k\tau}\|^2]\mathrm{d}s,
\end{align*}
where we apply Lemma \ref{lem:onesteperror_estimate} to deduce the last line.

In summary, we have that
\begin{align*}
    e^{2\lambda_1 t}\mathbb{E}[\|\hat{E}_t^{n,-k\tau}\|^2]&\leq  e^{-2\lambda_1 k\tau}\mathbb{E}[\|\xi-\eta\|^2]
    -\big(C_{f,g}-2C_f(\sqrt{h}\lambda_n+\sqrt{c_n})\sqrt{h}\big)\int_{-k\tau}^t e^{2\lambda_1 s}\mathbb{E}[\|\bar{E}_{s}^{n,-k\tau}\|^2]\mathrm{d}s\\
    &\leq e^{-2\lambda_1 k\tau}\mathbb{E}[\|\xi-\eta\|^2]
\end{align*}
because of the choice of stepsize $h$. Then the assertion follows.
\end{proof}
\begin{thm}\label{thm:main2}
Under Assumptions \ref{as:A} to \ref{as:lambda2}, for any $h\in(0,1)$ satisfying 
\begin{equation}\label{eqn:hmainchoice}
  2C_f(\sqrt{h}\lambda_n+\sqrt{c_n})\sqrt{h}\leq C_{f,g} \text{ and }     \big(5L_{f,g} \sqrt{h\lambda_n}(1+C_n h)+2C_f\sqrt{C_n}\big) \sqrt{h}\leq 2C_{f,g},
\end{equation}
the Galerkin numerical approximation \eqref{eq:con_RandM} admits a unique random period solution $\hat{X}^{n,*}_{t}\in L^2(\Omega;H)$ such that
    \begin{equation}\label{eqn:lim_sdenum}
    \lim_{k\to \infty}\mathbb{E}[\|\hat{X}^{n,-k\tau}_t(\xi)-\hat{X}^{n,*}_t\|^2]=0.
\end{equation}
%Moreover, it holds that the mild form of $X^{n,*}$ given by
%$$ \hat{X}_{t}^{n,*} = \int_{-\infty} ^t   S\big(t-\Lambda(s)\big) f_n\big(\Lambda(s),   \bar{X}_{s}^{n,-k\tau} \big)\mathrm{d}s+ 
%    \int_{-\infty} ^t   S\big(t-\Lambda(s)\big) g_n\big(\Lambda(s)\big)\mathrm{d}W(s)$$
%is in $\bigcap_{r\in (0,1)}\dot{H}^r$. 
\end{thm}
With Lemma \ref{lem:boundednessG} and Lemma \ref{lem:numsoldependence}, the proof is similar to the proof of Theorem 3.4 in \cite{rpsnumerics2017}.
\subsection{The convergence}\label{sec:convergence}
\begin{thm}\label{thm:error}
Under Assumption \ref{as:A} to Assumption \ref{as:fg},  and Assumption \ref{as:dissipative} to Assumption \ref{as:lambda2}, let $X^{-k\tau}_{\cdot}$ be the solution of SEE \eqref{eq:SPDE} with the initial condition $\xi\in L^{2}(\Omega, \mathcal{F}_{-k\tau}, \mathbb{P};
  \dot{H}^r)$ for some $r\in (0,1)$, and let $\hat{X}^{n,-k\tau}_{\cdot}$ be its numerical simulation defined by \eqref{eq:con_RandM} with the stepsize $h$ satisfying \eqref{eqn:hchoice}. Then for any any $\nu_1 \in (0,r/2]$, there exists a constant C, which depends on $\xi$, $A$, $f$, $g$ , $r$, $\nu_1$ and the uniform bounds of both $X^{-k\tau}_{\cdot}$ and $\hat{X}^{n,-k\tau}_{\cdot}$, such that
\begin{equation}\label{eqn:error}
    \sup_{k\in \mathbb{N}}\sup_{t\geq k\tau}\big(\mathbb{E}[\|X_t^{-k\tau}-\hat{X}_t^{n,-k\tau}\|^2]\big)^{1/2}\leq  C\big(h^{{\nu_1\wedge \kappa}}+\frac{1}{\sqrt{\lambda_n^r}}\big),
\end{equation}
{where $\nu_1\wedge \kappa$ represents the smaller between $\nu_1$ and $\kappa$. }
\end{thm}
\begin{proof}
From the mild form \eqref{eq:mild} and the continuous version \eqref{eq:con_RandM} for the Galerkin numerical approximation we derive that
\begin{align}
    \begin{split}
        &X_t^{-k\tau}-\hat{X}_t^{n,-k\tau}=S(t+k\tau)(\text{Id}-P_n)\xi + \int_{-k\tau}^t S(t-s)(\text{Id}-P_n)f(s,X_s^{-k\tau})\mathrm{d}s\\
        &\quad +\int_{-k\tau}^t S(t-s)\big(f_n(s,X_s^{-k\tau})-f_n\big(\Lambda(s),X_{\Lambda(s)}^{-k\tau}\big)\big)\mathrm{d}s\\
        &\quad +\int_{-k\tau}^t S(t-s)\big(f_n\big(\Lambda(s),X_{\Lambda(s)}^{-k\tau}\big)-f_n(\Lambda(s),\bar{X}_s^{n,-k\tau})\big)\mathrm{d}s\\
        &\quad +\int_{-k\tau}^t S(t-s)(S(s-\Lambda(s))-\text{Id})f_n(\Lambda(s),\bar{X}_s^{n,-k\tau})\mathrm{d}s\\
        &\quad +\int_{-k\tau}^t S(t-s)(\text{Id}-P_n)g(s,X_s^{-k\tau})\mathrm{d}W(s)\\
        &\quad+\int_{-k\tau}^t S(t-s)\big(g_n(s,X_s^{-k\tau})-g_n\big(\Lambda(s),X_{\Lambda(s)}^{-k\tau}\big)\big)\mathrm{d}W(s)\\
        &\quad +\int_{-k\tau}^t S(t-s)\big(g_n\big(\Lambda(s),X_{\Lambda(s)}^{-k\tau}\big)-g_n(\Lambda(s),\bar{X}_s^{n,-k\tau})\big)\mathrm{d}W(s)\\
        &\quad +\int_{-k\tau}^t S(t-s)(S(s-\Lambda(s))-\text{Id})g_n(\Lambda(s),\bar{X}_s^{n,-k\tau})\mathrm{d}W(s)=:\sum_{i=1}^9 J_i.
    \end{split}
\end{align}
It remains to estimate each of $\{J_i\}_{i=1}^9$ in $\mathbb{E}[\|\cdot\|^2]$ with a finite bound that is independent of $k$ and $t$. For $J_1$, we can get the following estimate based on Assumption \ref{as:A} and the condition on $\xi$:
\begin{align*}
   &\mathbb{E}[\|S(t+k\tau)(\text{Id}-P_n)\xi\|^2]\\ &=\mathbb{E}\big[ \sum_{i=n+1}^\infty e^{-2(t+k\tau)\lambda_i} (e_i,\xi)^2\big]=\mathbb{E}\big[ \sum_{i=n+1}^\infty \frac{e^{-2(t+k\tau)\lambda_i}}{\lambda_n^r} \lambda^r_n (e_i,\xi)^2\big]\\
   &\leq\frac{1}{\lambda_n^r}\mathbb{E}\big[ \sum_{i=1}^\infty \lambda^r_n (e_i,\xi)^2\big]=\frac{1}{\lambda_n^r}\mathbb{E}[\|A^{\frac{r}{2}}\xi\|^2].
\end{align*}
For $J_2$, by using the same decomposition for $x\in H$ as in $J_1$, the linear growth of $f$, and the uniform boundedness of $X^{-k\tau}_t$ (see Lemma \ref{lem:boundedness}), one can see that
\begin{align*}
       &\mathbb{E}\big[\big\|\int_{-k\tau}^t S(t-s)(\text{Id}-P_n)f(s,X_s^{-k\tau})\mathrm{d}s\big\|^2\big]\\
       &=\mathbb{E}\big[\big\|\int_{-k\tau}^t \sum_{i=n+1}^\infty e^{-2(t-s)\lambda_i} (e_i,f(s,X_s^{-k\tau}))\mathrm{d}s\big\|^2\big]\\
       &\leq \mathbb{E}\Big[\Big(\int_{-k\tau}^t e^{-2(t-s)\lambda_{n+1}} \Big\|\sum_{i=n+1}^\infty  (e_i,f(s,X_s^{-k\tau}))\Big\|\mathrm{d}s\Big)^2\Big]\\
       &\leq \int_{-k\tau}^t e^{-2\lambda_{n+1}(t-s)} \mathrm{d}s \int_{-k\tau}^t e^{-2\lambda_{n+1}(t-s)} \mathbb{E}[\|f(s,X_s^{-k\tau})\|^2]\mathrm{d}s\\
       &\leq 2L_{f,g}^2 \big(1+\sup_{k\in \mathbb{N}}\sup_{t\geq -k\tau}\mathbb{E}[\|X^{-k\tau}_t\|]\big)\frac{(1-e^{-2\lambda_{n+1}(t+k\tau)})^2}{\lambda_{n+1}^2}\\
       &\leq \frac{2}{\lambda_{n+1}^2} L_{f,g}^2 \big(1+\sup_{k\in \mathbb{N}}\sup_{t\geq -k\tau}\mathbb{E}[\|X^{-k\tau}_t\|]\big).
\end{align*}
To get the upper bound for $J_3$, one shall apply the H\"older inequality and Assumption \ref{as:fg}, and then make use of Lemma \ref{lem:solcontinuity},
\begin{align*}
           &\big(\mathbb{E}\big[\big\|\int_{-k\tau}^t S(t-s)\big(f_n(s,X_s^{-k\tau})-f_n\big(\Lambda(s),X_{\Lambda(s)}^{-k\tau}\big)\big)\mathrm{d}s\big\|^2\big]\big)^{1/2}\\
           &\leq \big(\mathbb{E}\big[\big\|\int_{-k\tau}^t S(t-s)\big(f_n(s,X_s^{-k\tau})-f_n\big(\Lambda(s),X_s^{-k\tau}\big)\big)\mathrm{d}s\big\|^2\big]\big)^{1/2}\\
           &+\big(\mathbb{E}\big[\big\|\int_{-k\tau}^t S(t-s)\big(f_n(\Lambda(s),X_s^{-k\tau})-f_n\big(\Lambda(s),X_{\Lambda(s)}^{-k\tau}\big)\big)\mathrm{d}s\big\|^2\big]\big)^{1/2}\\
           &\leq 2C_f \Big(\sqrt{1+\sup_{k\in\mathbb{N}}\sup_{s\geq -k\tau}\mathbb{E}[\|X^{-k\tau}_s\|^2]}{h^\kappa}+\sqrt{C_X(\nu_1,r)}h^{\nu_1}\Big)\int_{-k\tau}^t \|S(t-s)\|_{\mathcal{L}(H)}\mathrm{d}s\\
           &\leq \frac{1}{{\lambda_1}}2C_f \Big(\sqrt{1+\sup_{k\in\mathbb{N}}\sup_{s\geq -k\tau}\mathbb{E}[\|X^{-k\tau}_s\|^2]}h^{{\kappa}}+\sqrt{C_X(\nu_1,r)}{h^{\nu_1}}\Big),
\end{align*}
where to get the last line, we use the following estimate based on {Proposition \ref{prop:semiprop}},
\begin{align}\label{eqn:Sintegralestimate}
  \int_{-k\tau}^t \|S(t-s)\|_{\mathcal{L}(H)}\mathrm{d}s\leq \int_{-k\tau}^t e^{-{\lambda_1}(t-s)}\mathrm{d}s\leq \frac{1}{\alpha}.  
\end{align}
Regarding the term $J_4$, by Assumption \ref{as:fg} and the estimate \eqref{eqn:Sintegralestimate} one has that
\begin{align*}
    \mathbb{E}\big[\big\|\int_{-k\tau}^t S(t-s)\big(f_n\big(\Lambda(s),X_{\Lambda(s)}^{-k\tau}\big)-f_n(\Lambda(s),\bar{X}_s^{n,-k\tau})\big)\mathrm{d}s\big\|^2\big]\leq \frac{C_f^2}{{\lambda_1}^2} \sup_{k,s}\mathbb{E}[\|X_{s}^{-k\tau}-\hat{X}_s^{n,-k\tau}\|^2].
\end{align*}
For term $J_5$, following the estimate \eqref{eqn:Sdiffintegralestimate}
yields
\begin{align*}
    &\mathbb{E}\big[\big\|\int_{-k\tau}^tS(t-s)(S(s-\Lambda(s))-\text{Id})f_n(\Lambda(s),\bar{X}_s^{n,-k\tau})\mathrm{d}s\big\|^2\big]\\
    &\leq  2L_{f,g}^2 \big(1+\sup_{k\in\mathbb{N}}\sup_{t\geq -k\tau}\mathbb{E}[\|\hat{X}^{n,-k\tau}_t\|^2]\big) C_2(\nu_1)^2C_2(\nu_1)^2\frac{{\lambda_1}^{2(\nu_1-1)} \Gamma(1-\nu_1)^2}{4} h^{2\nu_1}.
\end{align*}

Directly applying the It\^o isometry and the estimate \eqref{eqn:Sdiffestimate},  we have the bound for $J_6$:
\begin{align*}
   & \mathbb{E}\big[\big\|\int_{-k\tau}^tS(t-s)(\text{Id}-P_n)g(s,X_s^{-k\tau})\mathrm{d}W(s)\big\|^2\big]\\
    &\leq 2L_{f,g}^2\big(1+\sup_{k\in\mathbb{N}}\sup_{t\geq -k\tau}\mathbb{E}[\|\hat{X}^{n,-k\tau}_t\|^2]\big)\int_{-k\tau}^t \|S(t-s)(\text{Id}-P_n)\|^2_{\mathcal{L}(H)}\mathrm{d}s\\
    &\leq \frac{2L_{f,g}^2}{\lambda_{n+1}^2}\big(1+\sup_{k\in\mathbb{N}}\sup_{t\geq -k\tau}\mathbb{E}[\|\hat{X}^{n,-k\tau}_t\|^2]\big).
\end{align*}
Through the It\^o isometry, Assumption \ref{as:A}, Assumption \ref{as:fg}, Lemma \ref{lem:solcontinuity} and a similar estimate as \eqref{eqn:Sintegralestimate}, one can derive the bound for $J_7$ as follows
\begin{align*}
           &\big(\mathbb{E}\big[\big\|\int_{-k\tau}^t S(t-s)\big(g_n(s,X_s^{-k\tau})-g_n\big(\Lambda(s),X_{\Lambda(s)}^{-k\tau}\big)\big)\mathrm{d}W(s)\big\|^2\big]\big)^{1/2}\\
           &\leq \big(\mathbb{E}\big[\big\|\int_{-k\tau}^t S(t-s)\big(g_n(s,X_s^{-k\tau})-g_n\big(\Lambda(s),X_{s}^{-k\tau}\big)\big)\mathrm{d}W(s)\big\|^2\big]\big)^{1/2}\\
           &\quad +\big(\mathbb{E}\big[\big\|\int_{-k\tau}^t S(t-s)\big(g_n(\Lambda(s),X_s^{-k\tau})-g_n\big(\Lambda(s),X_{\Lambda(s)}^{-k\tau}\big)\big)\mathrm{d}W(s)\big\|^2\big]\big)^{1/2}\\
           &\leq  2C_g\Big(\sqrt{1+\sup_{k\in\mathbb{N}}\sup_{s\geq -k\tau}\mathbb{E}[\|X^{-k\tau}_s\|^2]}{h^\kappa}+\sqrt{C_X(\nu_1,r)}h^{\nu_1}\Big)\big(\int_{-k\tau}^t \|S(t-s)\|^2_{\mathcal{L}(H)}\mathrm{d}s\big)^{1/2}\\
          & \leq  \frac{2C_g}{\sqrt{2{\lambda_1}}} \Big(\sqrt{1+\sup_{k\in\mathbb{N}}\sup_{s\geq -k\tau}\mathbb{E}[\|X^{-k\tau}_s\|^2]}h^{\kappa}+\sqrt{C_X(\nu_1,r)}{h^{\nu_1}}\Big).
\end{align*}
Regarding the term $J_8$, by the It\^o isometry, Assumption \ref{as:fg} and the estimate \eqref{eqn:Sintegralestimate} one has that
\begin{align*}
    \mathbb{E}\big[\big\|\int_{-k\tau}^t S(t-s)\big(g_n\big(\Lambda(s),X_{\Lambda(s)}^{-k\tau}\big)-g_n(\Lambda(s),\bar{X}_s^{n,-k\tau})\big)\mathrm{d}W(s)\big\|^2\big]\leq \frac{C_g^2}{2{\lambda_1}} \sup_{k,s}\mathbb{E}[\|X_{s}^{-k\tau}-\hat{X}_s^{n,-k\tau}\|^2].
\end{align*}

Finally, applying the It\^o isometry and the linear growth of $g$ in Assumption \ref{as:fg}, and following the estimate \eqref{eqn:SdiffWintegralestimate}, we have the bound for $J_9$ that
\begin{align*}
    &\mathbb{E}\big[\big\|\int_{-k\tau}^tS(t-s)(S(s-\Lambda(s))-\text{Id})g_n(\Lambda(s),\bar{X}_s^{n,-k\tau}))\mathrm{d}W(s)\big\|^2\big]\\
    &\leq 2L_{f,g}^2\big(1+\sup_{k\in\mathbb{N}}\sup_{t>-k\tau}\mathbb{E}[\|\hat{X}^{n,-k\tau}_t\|^2]\big)\int_{-k\tau}^t \|A^{\nu_1}S(t-s)A^{-\nu_1}(S(s-\Lambda(s))-\text{Id})\|^2_{\mathcal{L}(H)}\mathrm{d}s\\
  &\leq 2L_{f,g}^2\big(1+\sup_{k\in\mathbb{N}}\sup_{t>-k\tau}\mathbb{E}[\|\hat{X}^{n,-k\tau}_t\|^2]\big)C_1(\nu_1)^2 h^{2\nu_1}
C_2(\nu_1)^2\frac{(2{\lambda_1})^{2\nu_1-1} \Gamma(1-2\nu_1)^2}{2}.
\end{align*}
In total, we have that
\begin{align*}
     &\sup_{k\in \mathbb{N}}\sup_{t\geq k\tau}\big(\mathbb{E}[\|X_t^{-k\tau}-\hat{X}_t^{n,-k\tau}\|^2]\big)^{1/2}\leq \sum_{i=1}^9 \sup_{k\in \mathbb{N}}\sup_{t\geq k\tau}(\mathbb{E}[\|J_i\|^2])^{1/2} \\
     &\leq C(h^{\nu_1}+{h^{\kappa}}+\frac{1}{\sqrt{\lambda_n^r}}+\frac{1}{\lambda_n}+\frac{1}{\lambda_{n+1}})+\big(\frac{C_f}{{\lambda_1}}+\frac{C_g}{\sqrt{{\lambda_1}}}\big)\sup_{k\in \mathbb{N}}\sup_{t\geq k\tau}\big(\mathbb{E}[\|X_t^{-k\tau}-\hat{X}_t^{n,-k\tau}\|^2]\big)^{1/2}.
\end{align*}
Because of $\frac{C_f}{{\lambda_1}}+\frac{C_g}{\sqrt{{\lambda_1}}}<1$ from Assumption \ref{as:lambda2}, we can conclude the final assertion.
\end{proof}
Note in Theorem \ref{thm:error}, one can take $\nu_1=r/2$ to achieve the fastest convergence.
\begin{corr}\label{corr:error} Assume Assumption \ref{as:A} to Assumption \ref{as:lambda2}. Let $X^{*}_{t}$ be the random periodic solution of SEE \eqref{eq:SPDE} and $\hat{X}^{n,*}_{t}$ be the random period solution of the Galerkin numerical approximation with the stepsize $h$ satisfying \eqref{eqn:hmainchoice}. Consider approximating $\hat{X}^{n,*}_{\cdot}$ through the sequence $\{\hat{X}^{n,-k\tau}_{\cdot}(\xi)\}_{k}$ with $\xi\in L^{2}(\Omega, \mathcal{F}_{-k\tau}, \mathbb{P};
  \dot{H}^r)$ for $r\in (0,1)$. Then there exists a constant $C$, which depends on $A$, $f$, $g$ and $r$, such that 
\begin{align}\label{eq:error}
   \sup_{t} \big(\mathbb{E}[\|X^*_t-\hat{X}^{n,*}_t\big\|^2]\big)^{1/2}\leq C \big(h^{{\frac{r}{2}\wedge \kappa}}+\frac{1}{\sqrt{\lambda_n^r}}\big).
\end{align}
\end{corr}
Corollary \ref{corr:error} implies that the best order of convergence can be achieved is $1/2-\epsilon$ for an arbitrarily small $\epsilon>0$ if $\kappa=1/2$.  Moreover, as the mild form of $X^*_t$ is welled defined in $\bigcap_{r\in (0,1)}L^2(\Omega;\dot{H}^r)$ shown in Theorem \ref{thm:main1}, and $\dot{H}^{r_1}\subset \dot{H}^{r_2}$ for $r_1\geq r_2$, it is not surprised to observe that the order of convergence would be higher if we adopt the approximation sequence with initial condition in $L^{2}(\Omega; \dot{H}^r)$ under a higher value of $r$.
\section*{Declarations}
\begin{itemize}
\item Funding: This work is supported by the Alan Turing Institute for funding this work under EPSRC grant EP/N510129/1 and EPSRC for funding though the project EP/S026347/1, titled 'Unparameterised multi-modal data, high order signatures, and the mathematics of data science'.

\item Competing  interests:  Authors  are  required  to  disclose  financial  or  non-financial  interests  that  are  directly  or  indirectly  related  to  the  work submitted for publication.

\item Data Availability Statement: Data sharing not applicable to this article as no datasets were generated or analysed during the current study.
\end{itemize}

\bibliographystyle{unsrt}  


\end{document}